\documentclass[reqno]{amsart}

\usepackage[T1]{fontenc}
\usepackage{amsmath,amssymb,amsthm,mathtools}
\usepackage{enumitem}
\usepackage{microtype}
\usepackage[colorlinks=true,linkcolor=blue,citecolor=blue,urlcolor=blue]{hyperref}

\makeatletter
\renewcommand{\@settitle}{%
  \begin{center}%
    \baselineskip14\p@\relax
    \bfseries\@title
  \end{center}%
}
\makeatother

\newtheorem{thm}{Theorem}[section]
\newtheorem{prop}[thm]{Proposition}
\newtheorem{lem}[thm]{Lemma}
\newtheorem{cor}[thm]{Corollary}
\theoremstyle{definition}

\newtheorem{question}[thm]{Open question}
\theoremstyle{remark}
\newtheorem{rem}[thm]{Remark}

\newcommand{\R}{\mathbb{R}}
\newcommand{\E}{\mathbb{E}}
\renewcommand{\Pr}{\mathbb{P}}
\newcommand{\osc}{\operatorname{osc}}
\newcommand{\diam}{\operatorname{diam}}
\newcommand{\cE}{\mathcal{E}}
\newcommand{\cD}{\mathcal{D}}
\newcommand{\one}{\mathbf{1}}

\title[Consensus time for $\ell^p$ relaxation]{Consensus time for asynchronous $\ell^p$ relaxation: graph dependence}
\author{Chenyu Gan}
\address{Qiuzhen College, Tsinghua University, Beijing, China}
\email{gancy22@mails.tsinghua.edu.cn}
\subjclass[2020]{Primary 60J10; Secondary 05C81, 35J92, 90C25}
\keywords{asynchronous consensus, graph $p$-Laplacian, randomized coordinate
minimization, nonlinear Poincar\'e inequality, conductance, trees}
\date{}

\begin{document}
\raggedbottom

\begin{abstract}
We study the asynchronous $\ell^p$ relaxation introduced by Amir,
Nazarov, and Peres
\cite{amir2025convergencerateellpenergyminimization}: at each step, a uniformly
chosen vertex minimizes its incident $\ell^p$ energy.  For the profile $f_t$
after $t$ updates, let
\[
 \mathsf T_p(G,1/2):=
 \sup_{\lVert f_0\rVert_\infty\le1}
 \E\bigl[\min\{t\ge0:\osc(f_t)\le1/2\}\bigr].
\]
For $1<p<\infty$, we obtain graph-dependent estimates for boxes, trees, and
conductance expanders.  On the nearest-neighbor box $[L]^d$, for $d,L\ge2$
and $n=L^d$, the answer is, up to logarithmic factors,
$nd^{1/(p-1)}L^{p/(p-1)}$ for $1<p<2$ and $ndL^2$ for $p\ge2$.
On bounded-degree trees, an explicit rerooting-invariant parameter $T_G$
determines the answer up to logarithmic factors; for arbitrary trees it gives
upper and lower bounds that differ additionally by the maximum degree.  If
the volume conductance $h(G)\ge h_0>0$, then
$\mathsf T_p(G,1/2)=\Theta_{p,h_0}(n\log n)$ without a degree assumption.
At $p=\infty$, every connected graph satisfies
$\mathsf T_\infty(G,1/2)\ge c nD^2/\Delta$, where $D$ and $\Delta$ are its
diameter and maximum degree.
\end{abstract}

\maketitle

\section{Introduction}\label{sec:intro}

Amir, Nazarov, and Peres introduced the asynchronous $\ell^p$ relaxation
\cite{amir2025convergencerateellpenergyminimization}.  For fixed
$1<p<\infty$, they proved that the worst fixed-tolerance consensus time over
connected $n$-vertex graphs has polynomial exponent
\[
 \beta_p=\max\left\{\frac{2p}{p-1},3\right\},
\]
up to logarithmic factors.  More precisely, their bound is uniform over every
initial profile, and they construct graphs and profiles attaining the exponent;
they also give sharp refinements in terms of average degree.  In particular,
the worst-case behavior undergoes a phase transition at $p=3$.

The quadratic case already shows why a worst-case graph bound does not tell the
whole story.  When $p=2$, the update is ordinary averaging and the dynamics are
linear.  If $\mu_1(G)$ is the first positive eigenvalue of the normalized
Laplacian and $t_{\mathrm{rel}}(G)=\mu_1(G)^{-1}$ is the random-walk relaxation
time, then the spectral estimates of Elboim, Peres, and Peretz
\cite{elboim2024asynchronousdegrootdynamics}, expressed in single-vertex
updates, imply that the fixed-tolerance worst-profile consensus time lies
between
\[
 c n t_{\mathrm{rel}}(G)
 \quad\text{and}\quad
 C n t_{\mathrm{rel}}(G)\log(en);
\]
thus the relaxation time characterizes the polynomial scale of quadratic
consensus, up to one logarithmic factor.  Since $t_{\mathrm{rel}}(G)$ depends
strongly on the graph, consensus can be much faster than the $n^4$ worst case:
an expander graph has an $O(n\log n)$ upper bound.

This contrast motivates a graph-specific analysis beyond the linear case.  We
ask how dimension, branching, conductance, and diameter control consensus when
the spectral description is no longer available.  Our main results answer this
question for three basic families.  For boxes we determine the correct scale up
to logarithmic factors, with different dimension dependence below and above
$p=2$.  For trees we give upper and lower bounds through a rerooting-invariant
recursive parameter $T_G$.  For every fixed $1<p<\infty$, conductance expanders
have consensus time $\Theta(n\log n)$ without a maximum-degree assumption.  We
also prove an every-graph diameter lower bound at the endpoint $p=\infty$.
Beyond these family results, for every $2\le p<\infty$ the worst-profile consensus time
on an arbitrary graph is at most the corresponding quadratic consensus time,
up to a $p$-dependent constant and one logarithmic factor; see
Corollary~\ref{cor:universal-p-comparison} below.

\subsection{The dynamics and global notation}

Let $G=(V,E)$ be a finite connected simple graph with $n=|V|\ge2$,
maximum degree $\Delta$, and diameter $D=\diam(G)$.  We write $v\sim u$
when $\{u,v\}\in E$ and $\deg(u)$ for the degree of $u$.  A profile is a
function $f:V\to\R$.  For $1<p<\infty$, define the local $p$-mean of the
neighboring values $a=(a_v)_{v\sim u}$ by
\begin{equation}\label{eq:intro-local-mean}
 m_p(a)=\mathop{\rm argmin}_{x\in\R}\sum_{v\sim u}|x-a_v|^p.
\end{equation}
The minimizer is unique by strict convexity.  At the endpoint $p=\infty$,
we use the midrange rule
\begin{equation}\label{eq:intro-midrange}
 m_\infty(a)=\frac{\max_{v\sim u}a_v+\min_{v\sim u}a_v}{2}.
\end{equation}

Let $U_1,U_2,\ldots$ be independent vertices chosen uniformly from $V$.
Starting from a deterministic profile $f_0$, the asynchronous $\ell^p$ relaxation dynamics is defined explicitly by
\begin{equation}\label{eq:intro-dynamics}
 f_t(x)=
 \begin{cases}
 m_p\bigl((f_{t-1}(v))_{v\sim U_t}\bigr),&x=U_t,\\
 f_{t-1}(x),&x\ne U_t,
 \end{cases}
 \qquad t\ge1.
\end{equation}
Thus one selected vertex is relaxed at each discrete time; no vertices are
frozen.  We write $f^{\#u}$ for the profile obtained by applying this update
at a specified vertex $u$, and $Sf$ for the synchronous profile with
$(Sf)(u)=m_p((f(v))_{v\sim u})$ at every $u$.

The oscillation, the consensus time at tolerance $\epsilon>0$, and its
worst-profile expectation are
\begin{equation}\label{eq:intro-consensus-time}
 \begin{gathered}
 \osc(f)=\max_{u\in V}f(u)-\min_{u\in V}f(u),\qquad
 \tau_p(\epsilon;f_0)=\min\{t\ge0:\osc(f_t)\le\epsilon\},\\
 \mathsf T_p(G,\epsilon)=
 \sup_{\lVert f_0\rVert_\infty\le1}\E[\tau_p(\epsilon;f_0)].
 \end{gathered}
\end{equation}
All probabilities and expectations, denoted by $\Pr$ and $\E$, refer to the
random choices $(U_t)$.  All logarithms are natural.  Constants $c,C>0$ are
absolute, while subscripts record their permitted parameter dependence; for
example, $c_p$ and $C_{p,\eta}$ may depend only on $p$ and on $(p,\eta)$,
respectively.  The notation $A\lesssim_p B$ means $A\le C_pB$.

Our graph-family results concern fixed $1<p<\infty$.  The rule
\eqref{eq:intro-midrange} is the asynchronous Lipschitz-learning dynamics,
and we also give a separate diameter lower bound for that endpoint.

We now state the graph-family results in detail.

\subsection{Boxes}

Take $G$ to be the nearest-neighbor box with vertex set
\[
 V=[L]^d=\{0,\ldots,L-1\}^d,
\]
where $d,L\ge2$ are integers, and write $n=L^d$.  Define
\begin{equation}\label{eq:intro-box-scale}
 U_p(d,L)=
 \begin{cases}
  nd^{1/(p-1)}L^{p/(p-1)},&1<p<2,\\
  ndL^2,&p\ge2.
 \end{cases}
\end{equation}
Theorem~\ref{thm:box-bounds} and
Corollary~\ref{cor:box-polylogarithmic-comparison} show that
\begin{equation}\label{eq:intro-box-comparison}
 c_p\frac{U_p(d,L)}{\log(en)}
 \le \mathsf T_p(G,1/2)
 \le C_pU_p(d,L)\log(en),\qquad p\ne2,
\end{equation}
whereas
\begin{equation}\label{eq:intro-box-comparison-two}
 cU_2(d,L)
 \le \mathsf T_2(G,1/2)
 \le CU_2(d,L)\log(en).
\end{equation}
Thus boxes are sharp up to logarithmic factors: the displayed upper/lower
ratio is $O_p((\log n)^2)$ for $p\ne2$ and $O(\log n)$ for $p=2$.  The
detailed proved lower bound is
\begin{equation}\label{eq:intro-box-detailed-lower}
 \mathsf T_p(G,1/2)\ge c_pn
 \begin{cases}
 \max\{d^{1/(p-1)-1}L^{p/(p-1)},d\log L\},&1<p<2,\\
 dL^2,&p=2,\\
 \max\{d\log L,L^2\},&p>2.
 \end{cases}
\end{equation}
The terms in a maximum may be realized by different initial profiles.  The restriction
$d\ge2$ is substantive for the subquadratic slow profile: on a path, a
two-neighbor $p$-mean is always the midpoint, so the displayed
$d$-dimensional mechanism does not extend to $d=1$.

There are related one-dimensional antecedents in Amir et al.
\cite[Theorem~5.1 and Claim~5.5]{amir2025convergencerateellpenergyminimization}:
their cycle example gives a cubic lower bound for every $p>1$, and their
second proof uses a quadratic slow profile on a line segment.  The new feature
of the present box construction is the $d$-dimensional inactive-neighbor
mechanism and its explicit dimension dependence.

\subsection{Trees}

Let $G$ now be a tree and root it at $r$.  A vertex with no children has
$w_r(v)=1$; recursively, a nonleaf has
\[
 w_r(v)=\left(\sum_{z\text{ child of }v}
 (w_r(z)+1)^{p-1}\right)^{1/(p-1)}.
\]
Writing $\operatorname{Path}(x,r)$ for the unique path from $x$ to $r$, set
\[
 T_G(r)=\max_{x\in V}
 \sum_{\substack{v\in\operatorname{Path}(x,r)\\v\ne r}}w_r(v),
 \qquad T_G=\min_{r\in V}T_G(r).
\]
Thus $T_G$ is independent of a distinguished drawing of the tree.  The
dependence on the fixed exponent is suppressed in the notation: throughout
the tree results, $T_G=T_G^{(p)}$.

At tolerance $1/2$, Theorem~\ref{thm:tree-gap} and
Corollary~\ref{cor:tree-consensus-upper} give the upper bound
\[
 \mathsf T_p(G,1/2)
 \le C_pnT_G\log(en)
 \begin{cases}
  \log(en)^{1/(p-1)},&1<p<2,\\
  1,&p\ge2,
 \end{cases}
\]
while Theorem~\ref{thm:tree-lower} and
Corollary~\ref{cor:tree-worst-lower} give
\[
 \mathsf T_p(G,1/2)
 \ge c_pn\max\left\{\log n,\frac{T_G}{\Delta}\right\},
\]
for every finite $p>1$.
In particular, if $\Delta\le\Delta_0$ for a fixed constant $\Delta_0$, then
$T_G^{(p)}$ characterizes the consensus scale up to one factor $\log(en)$ for
$p\ge2$, and up to $C_{p,\Delta_0}\log(en)^{1+1/(p-1)}$ for $1<p<2$.
For unrestricted-degree trees the displayed estimates differ by a factor
$\Delta$, which can be polynomial in $n$; we therefore do not claim a
polylogarithmic characterization for arbitrary trees.

The recursive parameter is explicit on spherically symmetric trees.  Suppose
the tree has layers $0,\ldots,k$, every layer-$i$ vertex has $d_i\ge1$ children for
$0\le i<k$, and the distinguished root has $d_0\ge2$ children.  Then the
natural root minimizes $T_G(r)$ and
Proposition~\ref{prop:symmetric-tree-geometry} gives, with
\[
 S=\sum_{1\le i\le j\le k-1}(d_i\cdots d_j)^{1/(p-1)},
\]
the explicit comparison $1+S\le T_G\le1+2S$ (the sum is empty for $k=1$).

As a concrete instance, Amir et al.
\cite[Theorem~5.7]{amir2025convergencerateellpenergyminimization} previously
treated the tree $H_m$ formed from a path of length $2m$ with $2m$ leaves
attached at each endpoint.  Rooting at the midpoint gives the symmetric
branching sequence
\[
 k=m+1,\qquad d_0=2,\qquad d_1=\cdots=d_{m-1}=1,
 \qquad d_m=2m.
\]
Consequently,
\[
 S=\frac{m(m-1)}2+m(2m)^{1/(p-1)}.
\]
Since $|V(H_m)|=6m+1$ and $\Delta(H_m)=2m+1$, for $1<p\le2$ the comparison
above gives $T_{H_m}\asymp_p m^{p/(p-1)}$.  Thus our upper bound recovers
their sharp $m^{(2p-1)/(p-1)}$ exponent up to logarithmic factors.  On this
example our general propagation lower bound gives only order
$m^{p/(p-1)}$, a factor of order $m$ smaller, and hence does not subsume their
special-tree lower bound.

\subsection{Conductance expanders}

For $A\subseteq V$, let
$\operatorname{vol}(A)=\sum_{u\in A}\deg(u)$ and let $\partial A$ be the set
of edges with exactly one endpoint in $A$.  The volume conductance is
\[
 h(G)=\min_{0<\operatorname{vol}(A)\le\operatorname{vol}(V)/2}
 \frac{|\partial A|}{\operatorname{vol}(A)}.
\]
Theorem~\ref{thm:expander-consensus} proves that, for every fixed
$1<p<\infty$, every family satisfying $h(G)\ge h_0>0$ has
\begin{equation}\label{eq:intro-expander}
 \mathsf T_p(G,1/2)
 =\Theta_{p,h_0}(n\log n).
\end{equation}
There is no maximum-degree assumption.  The coarea inequality used to obtain
nonlinear Poincar\'e coercivity is a graph $p$-Cheeger argument of the standard
type studied by Keller and Mugnolo \cite{KellerMugnolo2016} and by Tudisco and
Hein \cite{TudiscoHein2018}; the contribution here is its conversion, through
the exact one-site drop and the coupon obstruction, into a sharp consensus-time
statement.

\subsection{The Lipschitz-learning endpoint}

When $p=\infty$, \eqref{eq:intro-dynamics} uses the midrange
\eqref{eq:intro-midrange}.  Amir et al.
\cite[Theorem~1.3(a)]{amir2025convergencerateellpenergyminimization} prove that,
for $0<\epsilon\le1$,
\begin{equation}\label{eq:intro-lipschitz-upper}
 \mathsf T_\infty(G,\epsilon)
 \le n(\log n+1)(D+1)^2\log(4/\epsilon).
\end{equation}
Here the factor $\log(4/\epsilon)$ results from rescaling their
$[0,1]^V$ normalization to $[-1,1]^V$.  Their Theorem~1.3(b) also shows that,
for every $N\ge L+1$ and $L\ge2$, there is a connected graph with at most $N$ vertices
and diameter $L$, together with an initial profile, whose consensus time is at
least $cNL^2$ for every update sequence.  We use the usual edge-count graph
distance, under which diameter $L$ requires at least $L+1$ vertices; this
explains the admissible range stated here.  Our
Theorem~\ref{thm:lipschitz-diameter-lower} proves the complementary
every-graph statement: for every finite connected graph there exists
$f_0\in[-1,1]^V$ such that
\begin{equation}\label{eq:intro-lipschitz-lower}
 \E\tau_\infty(\epsilon;f_0)
 \ge \frac{1}{16e\pi^2}\frac{nD^2}{\Delta},
 \qquad 0<\epsilon\le1.
\end{equation}
Thus our result strengthens the existential graph quantifier in the earlier
diameter lower bound to every connected graph, at the cost of the factor
$\Delta$.  At fixed tolerance, the $nD^2$ polynomial scale is matched up to
that factor and a logarithm in $n$; in particular, for bounded-degree graphs
the diameter dependence is sharp up to that logarithm.

\subsection{Common framework and organization}

To define the two parameters used in the general upper bound, fix
$1<p<\infty$ and put
\[
 \cE_p(f)=\sum_{\{u,v\}\in E}|f(u)-f(v)|^p,
 \qquad
 \cD_p(f,u)=\cE_p(f)-\cE_p(f^{\#u}).
\]
The exact relaxation gap and the degree-weighted scalar Poincar\'e constant are
\begin{align*}
 \lambda_p(G)
 &=\inf_{\cE_p(f)>0}
 \frac{\sum_{u\in V}\cD_p(f,u)}{\cE_p(f)},\\
 \Phi_p(G)
 &=\sup_{f\ \mathrm{nonconstant}}
 \frac{\inf_{c\in\R}\sum_{u\in V}\deg(u)|f(u)-c|^p}{\cE_p(f)}.
\end{align*}
The first parameter is exactly the multiplicative one-step contraction
coefficient for randomized exact coordinate minimization, since
\[
 \E[\cE_p(f_1)\mid f_0=f]
 =\cE_p(f)-\frac1n\sum_{u\in V}\cD_p(f,u).
\]
Theorem~\ref{thm:variational-upper} bounds the consensus time using
$\lambda_p(G)$, and Theorem~\ref{thm:nonlinear-poincare-comparison} proves
\[
 \lambda_p(G)\gtrsim_p
 \begin{cases}
  \Phi_p(G)^{-1/(p-1)},&1<p<2,\\
  \Phi_p(G)^{-2/p},&p\ge2.
 \end{cases}
\]
The exact one-site drop avoids the singular quotients that arise when
neighboring values coincide.  The lower bounds used later are the unnumbered
coupon observation and the profile-sensitive propagation theorem in
Section~\ref{sec:lower}; the latter also supplies the Lipschitz-learning
lower bound \eqref{eq:intro-lipschitz-lower}.

For $2\le p<\infty$, nonlinear Poincar\'e extrapolation and the quadratic spectral
lower bound further give the graph-universal comparison
\begin{equation}\label{eq:intro-universal-p-comparison}
 \mathsf T_p(G,\epsilon)
 \le C_p\log\!\left(\frac{en}{\epsilon}\right)
       \mathsf T_2(G,\epsilon),
 \qquad 0<\epsilon\le\frac12;
\end{equation}
see Corollary~\ref{cor:universal-p-comparison} and
\eqref{eq:Tp-vs-T2}.  The suprema on the two sides may be attained by different
initial profiles, so this is a worst-profile partial answer, rather than an
answer to the same-profile comparison in
\cite[Question~6.1]{amir2025convergencerateellpenergyminimization}.

Section~\ref{sec:related} discusses previous work.
Section~\ref{sec:preliminaries} records the model and the elementary
properties of the local means.  Sections~\ref{sec:upper} and
\ref{sec:lower} give the general upper and lower mechanisms actually used;
Section~\ref{sec:lower} also proves the endpoint diameter lower bound.
Sections~\ref{sec:boxes}, \ref{sec:trees}, and \ref{sec:expanders} treat the
three finite-$p$ graph families in that order.  Section~\ref{sec:conclusion}
summarizes the remaining open problems.

\section{Related work}\label{sec:related}

At $p=2$, the update is ordinary averaging, so the process is the
vertex-asynchronous DeGroot dynamics.  DeGroot introduced the synchronous
model for opinion pooling \cite{DeGroot1974}, and Golub and Jackson studied
network influence and the ``wisdom of crowds'' in that setting
\cite{GolubJackson2010}.  For asynchronous vertex updates, Elboim et al.
\cite{elboim2024asynchronousdegrootdynamics} proved spectral upper and
lower bounds which, in our discrete-update normalization, place the
worst-profile consensus time at scale $n t_{\mathrm{rel}}(G)$ up to a
logarithmic factor.  Randomized gossip, studied by Boyd, Ghosh, Prabhakar, and
Shah \cite{BoydGhoshPrabhakarShah2006}, instead averages the endpoints of an
activated edge, so its spectral averaging-time estimates do not directly apply
to our all-neighbor vertex update.

For general $1<p<\infty$, Amir et al.
\cite{amir2025convergencerateellpenergyminimization} introduced the
boundary-free asynchronous $\ell^p$ relaxation and determined, up to
logarithmic factors, the worst fixed-tolerance consensus-time exponent
$\max\{2p/(p-1),3\}$ over connected $n$-vertex graphs, with refinements in
terms of average degree.  Their work also contains the local-residual descent
estimate and the cycle, line-segment, and two-ended-tree lower-bound mechanisms
used as antecedents here.  We instead seek graph-specific bounds: our
exact-drop coefficient $\lambda_p(G)$ and its comparison with the nonlinear
Poincar\'e constant $\Phi_p(G)$ yield results for boxes, arbitrary trees, and
conductance expanders, as well as a worst-profile comparison with $p=2$ for
$p\ge2$.  Gan, Peres, and Zuo
\cite{gan2025convergencerateellprelaxationgraph} subsequently studied the
Dirichlet version, with a frozen nonempty boundary and convergence to the
$p$-harmonic extension, and proved sharp worst-case and
average-degree-refined approximation bounds.  Their boundary-value results
and our boundary-free consensus estimates are complementary.

At the endpoint $p=\infty$, the update is the midrange rule from Lipschitz
learning.  Amir et al.
\cite{amir2025convergencerateellpenergyminimization} proved an
$O(n\log n\,D^2)$ fixed-tolerance upper bound and constructed graphs attaining
the $nD^2$ scale.  Our diameter lower bound applies to every connected graph,
at the cost of a factor of the maximum degree.  Peres, Schramm, Sheffield, and
Wilson connected infinity-harmonic functions to tug-of-war games
\cite{PeresSchrammSheffieldWilson2009}; Oberman developed a convergent scheme
for absolutely minimizing Lipschitz extensions \cite{Oberman2005}; and Kyng,
Rao, Sachdeva, and Spielman developed graph algorithms for Lipschitz learning
\cite{KyngRaoSachdevaSpielman2015}.

These connections suggest potential applications of the dynamics.  The
boundary-free $\ell^p$ process is a nonlinear asynchronous consensus rule and
an exact randomized coordinate-minimization scheme for graph $\ell^p$
energies, which may be relevant to distributed agreement and opinion
aggregation.  With frozen labels, closely related graph $p$-Laplacian methods
are used for image and data processing by Elmoataz, Toutain, and Tenbrinck
\cite{ElmoatazToutainTenbrinck2015}, and for image processing and clustering by
Elmoataz, Desquesnes, and Toutain
\cite{ElmoatazDesquesnesToutain2017}.  Slep\v{c}ev and Thorpe analyzed
$p$-Laplacian regularization in semi-supervised learning
\cite{SlepcevThorpe2019}, while Flores, Calder, and Lerman studied
$\ell^p$-based semi-supervised learning on graphs
\cite{FloresCalderLerman2022}.  At $p=\infty$, the midrange dynamics is a
natural iterative rule for Lipschitz learning and graph interpolation.
Our results concern convergence to a constant, so their dependence on
dimension, branching, conductance, and diameter provides prospective guidance
for those boundary-value and regularized problems rather than an
iteration-complexity guarantee for them.

\section{Model, notation, and preliminaries}\label{sec:preliminaries}

Throughout, $G=(V,E)$ is a finite connected simple graph with
$n=|V|\ge2$, maximum degree $\Delta$, and diameter
$D=\diam(G)$.  Each edge in a sum over $E$ is counted once.

\subsection{Local means and update operators}

For $1<p<\infty$ and a nonempty vector $y=(y_1,\ldots,y_d)$, define its
local $p$-mean by
\begin{equation}\label{eq:p-mean}
 m_p(y)=\mathop{\rm argmin}_{x\in\R}\sum_{i=1}^d |x-y_i|^p.
\end{equation}
Strict convexity makes the minimizer unique.
At the endpoint $p=\infty$, we use the midrange convention
\begin{equation}\label{eq:infinity-mean}
 m_\infty(y)=\frac{\max_i y_i+\min_i y_i}{2}.
\end{equation}

\begin{lem}\label{lem:mean-properties}
For $p\in(1,\infty]$, the map $m_p$ has the following properties.
\begin{enumerate}[label=(\alph*)]
\item It is continuous and lies between the smallest and largest coordinate.
\item It is positively affine equivariant: for every $\alpha>0$ and
$a\in\R$,
\[
 m_p(\alpha y+a\one)=\alpha m_p(y)+a.
\]
\item It is coordinatewise monotone.  Consequently,
\[
 |m_p(y)-m_p(z)|\le \lVert y-z\rVert_\infty
\]
whenever $y$ and $z$ have the same number of coordinates.
\end{enumerate}
\end{lem}

\begin{proof}
The assertions follow directly from \eqref{eq:infinity-mean} when
$p=\infty$.  Suppose now that $p<\infty$.
Set $\varphi_p(s)=|s|^{p-2}s$, with $\varphi_p(0)=0$.  The minimizer in
\eqref{eq:p-mean} is the unique zero of the continuous strictly increasing
function
\[
 x\longmapsto \sum_{i=1}^d\varphi_p(x-y_i).
\]
This description proves continuity and the interval assertion.  Translation
covariance follows by shifting both the argument and every coordinate by
$a$.  Moreover,
\[
 \sum_{i=1}^d|x-\alpha y_i|^p
 =\alpha^p\sum_{i=1}^d|x/\alpha-y_i|^p,
 \qquad \alpha>0,
\]
so uniqueness of the minimizer gives positive homogeneity.  Combining the
two identities proves positive affine equivariance.  If $y_i\le z_i$ for
every $i$, then the function associated with $y$ is pointwise at least the
function associated
with $z$, so their zeros satisfy $m_p(y)\le m_p(z)$.  Finally, if
$\delta=\lVert y-z\rVert_\infty$, then
$y-\delta\one\le z\le y+\delta\one$; monotonicity and translation covariance
give the last assertion.
\end{proof}

For a profile $f:V\to\R$ and $u\in V$, let $f^{\#u}$ denote the profile
obtained by updating only $u$:
\begin{equation}\label{eq:single-update}
 f^{\#u}(x)=
 \begin{cases}
 m_p\bigl((f(v))_{v\sim u}\bigr),&x=u,\\
 f(x),&x\ne u.
 \end{cases}
\end{equation}
The synchronous update $Sf$ is defined by
\begin{equation}\label{eq:synchronous-update}
 (Sf)(u)=m_p\bigl((f(v))_{v\sim u}\bigr),\qquad u\in V.
\end{equation}
No boundary vertices are frozen.

Let $(U_t)_{t\ge1}$ be independent uniformly distributed vertices of $V$.
Starting from a deterministic profile $f_0$, the asynchronous dynamics is
\begin{equation}\label{eq:asynchronous-update}
 f_t=f_{t-1}^{\#U_t},\qquad t\ge1.
\end{equation}
All probabilities and expectations below refer to this update sequence.  We
write
\[
 \osc(f)=\max_{u\in V}f(u)-\min_{u\in V}f(u)
\]
and, for $\epsilon>0$, define the stopping time
\begin{equation}\label{eq:consensus-time}
 \tau_p(\epsilon;f_0)
 =\min\{t\ge0:\osc(f_t)\le\epsilon\},
\end{equation}
with the convention that the minimum is $\infty$ if the set is empty.

\begin{lem}\label{lem:oscillation-energy}
Every single-vertex update takes its new value in the range of the old
profile.  Hence $\osc(f_t)$ is nonincreasing.  Moreover, for
$1<p<\infty$,
\begin{equation}\label{eq:energy-controls-oscillation}
 \cE_p(f):=\sum_{\{u,v\}\in E}|f(u)-f(v)|^p
 \ge \frac{\osc(f)^p}{D^{p-1}}.
\end{equation}
\end{lem}

\begin{proof}
The first assertion follows from Lemma~\ref{lem:mean-properties}.  For the
second, join vertices attaining the maximum and minimum by a shortest path
 of length $\ell\le D$.  H\"older's inequality along this path gives
\[
 \osc(f)^p\le \ell^{p-1}
 \sum_{e\text{ on the path}}|\nabla_e f|^p
 \le D^{p-1}\cE_p(f).
\]
\end{proof}

\section{General upper-bound framework}\label{sec:upper}

\subsection{The exact relaxation gap}

For $1<p<\infty$, define the energy drop at $u$ by
\begin{equation}\label{eq:energy-drop}
 \cD_p(f,u)=\cE_p(f)-\cE_p(f^{\#u}).
\end{equation}
Only edges incident to $u$ change, and \eqref{eq:p-mean} shows that
$\cD_p(f,u)\ge0$.  The exact relaxation gap is
\begin{equation}\label{eq:relaxation-gap}
 \lambda_p(G)=
 \inf_{\cE_p(f)>0}
 \frac{\sum_{u\in V}\cD_p(f,u)}{\cE_p(f)}.
\end{equation}
This definition avoids singular expressions when neighboring values coincide.

\begin{prop}\label{prop:gap-properties}
For every finite connected graph and every $1<p<\infty$,
\[
 0<\lambda_p(G)\le n.
\]
The quotient in \eqref{eq:relaxation-gap} is invariant under translation of
$f$ and multiplication of $f$ by a nonzero scalar.
\end{prop}

\begin{proof}
Translation invariance and homogeneity of degree $p$ follow directly from the
definitions.  Also $0\le\cD_p(f,u)\le\cE_p(f)$, so the upper bound is
immediate.

Fix a vertex $o$.  By the two invariances, the infimum can be taken over
\[
 \mathcal K=\{f: f(o)=0,\ \cE_p(f)=1\}.
\]
This set is compact: along a path from $o$ to any vertex, every edge
increment has absolute value at most one, so all coordinates are uniformly
bounded.  Lemma~\ref{lem:mean-properties} implies that
$f\mapsto\sum_u\cD_p(f,u)$ is continuous.  If it vanished at some
$f\in\mathcal K$, then every vertex would already equal the $p$-mean of its
neighbors.  At a vertex where $f$ is maximal, this is possible only when all
neighbors have the same maximal value.  Connectivity would then make $f$
constant, contradicting $\cE_p(f)=1$.  The continuous numerator therefore
has a strictly positive minimum on $\mathcal K$.
\end{proof}

\begin{thm}[Variational consensus-time upper bound]\label{thm:variational-upper}
Let $1<p<\infty$, let $f_0\in[-1,1]^V$, and let
$0<\epsilon\le1$.  Then
\begin{equation}\label{eq:variational-upper}
 \E\tau_p(\epsilon;f_0)
 \le \frac{n}{\lambda_p(G)}
 \left[
 \log\!\left(
 \frac{2^p|E|D^{p-1}}{\epsilon^p}
 \right)+3
 \right].
\end{equation}
\end{thm}

\begin{proof}
Conditioning on $f_t$ and using \eqref{eq:relaxation-gap},
\[
 \E[\cE_p(f_{t+1})\mid f_t]
 =\cE_p(f_t)-\frac1n\sum_{u\in V}\cD_p(f_t,u)
 \le\left(1-\frac{\lambda_p(G)}n\right)\cE_p(f_t).
\]
Consequently,
\[
 \E\cE_p(f_t)
 \le e^{-\lambda_p(G)t/n}\cE_p(f_0)
 \le 2^p|E|e^{-\lambda_p(G)t/n}.
\]
By Lemma~\ref{lem:oscillation-energy} and Markov's inequality,
\begin{equation}\label{eq:tail-upper}
 \Pr(\tau_p(\epsilon;f_0)>t)
 \le A e^{-xt},
 \qquad
 A=\frac{2^p|E|D^{p-1}}{\epsilon^p},
 \quad x=\frac{\lambda_p(G)}n.
\end{equation}
Here $0<x\le1$ by Proposition~\ref{prop:gap-properties}, and $A>1$ under
the stated hypotheses.  Split the tail sum
$\E\tau=\sum_{t\ge0}\Pr(\tau>t)$ at
$T=\lceil x^{-1}\log A\rceil$.  The part below $T$ is at most $T$, while
the geometric tail is at most
$(1-e^{-x})^{-1}\le1+x^{-1}$.  This gives
\eqref{eq:variational-upper}.
\end{proof}

\subsection{Exact drops and a nonlinear Poincar\'e comparison}

For a profile $f$ and a vertex $u$, put
\begin{equation}\label{eq:local-residual}
 B_f(u)=\sum_{v\sim u}|f(u)-f(v)|^{p-2}(f(u)-f(v)).
\end{equation}
The summand is zero when the difference is zero.  Also set
\[
 A_f(u)=\sum_{v\sim u}|f(u)-f(v)|^{p-2}.
\]
When $1<p<2$, this is an extended-real quantity: a zero difference makes
$A_f(u)=\infty$.  We use the extended-real conventions
$0^{p-2}=+\infty$ and $x/(+\infty)=0$ for every finite $x$; in particular,
$B_f(u)^2/A_f(u)$ and $|B_f(u)|/A_f(u)$ are then zero.  When
$p=2$, every summand in $A_f(u)$ is interpreted as one, so $A_f(u)=\deg(u)$.
When $p>2$, zero differences contribute zero, and the ratio is defined to be
zero if $A_f(u)=B_f(u)=0$.

One-site energy-decrease estimates in terms of such local residuals, including
the characteristic residual/degree scale, were developed in
\cite[Lemma~4.1 and Claim~4.2]{amir2025convergencerateellpenergyminimization}.
The following zero-safe formulation records the two complementary regimes
needed to compare the normalized global gap $\lambda_p(G)$ with
$\Phi_p(G)$.

\begin{lem}[Local energy-drop comparison]\label{lem:local-drop-comparison}
Let $q=p/(p-1)$.  There is $c_p>0$ such that, for every profile $f$ and
vertex $u$ of degree $d$,
\begin{equation}\label{eq:local-drop-comparison}
 \cD_p(f,u)\ge c_p
 \begin{cases}
 \displaystyle\max\left\{
  \frac{B_f(u)^2}{A_f(u)},
  \frac{|B_f(u)|^q}{d^{1/(p-1)}}\right\},&1<p<2,\\[3mm]
 \displaystyle\min\left\{
  \frac{B_f(u)^2}{A_f(u)},
  \frac{|B_f(u)|^q}{d^{1/(p-1)}}\right\},&p\ge2.
 \end{cases}
\end{equation}
All expressions are understood by the conventions above.  In particular,
the estimate remains valid when some incident edge differences vanish.
\end{lem}

\begin{proof}
Write the incident differences as $a_i=f(u)-f(v_i)$, and write
$g=f(u)-f^{\#u}(u)$.  Thus
\begin{equation}\label{eq:g-equation}
 \sum_{i=1}^d\varphi_p(a_i-g)=0,
 \qquad B:=B_f(u)=\sum_{i=1}^d\varphi_p(a_i).
\end{equation}
The signs of $B$ and $g$ agree.  We first prove explicitly the drop estimate
used below.  Define
\begin{equation}\label{eq:midpoint-constant}
 c_p^\ast=\inf_{s\in\R}
 \frac{\varphi_p(s-1/2)-\varphi_p(s-1)}
      {\varphi_p(s)-\varphi_p(s-1)}>0.
\end{equation}
Indeed, the quotient is positive and continuous because $\varphi_p$ is
continuous and strictly increasing, and it tends to $1/2$ as
$s\to\pm\infty$.  Its infimum is therefore positive.  Homogeneity gives,
for every $x\in\R$ and $h>0$, the scalar inequality
\begin{equation}\label{eq:midpoint-increment}
 \varphi_p(x-h/2)-\varphi_p(x-h)
 \ge c_p^\ast[\varphi_p(x)-\varphi_p(x-h)].
\end{equation}
This includes the cases in which any of $x,x-h/2,x-h$ is zero, because
$\varphi_p(0)=0$ and no negative power is used.

If $g=0$, then \eqref{eq:g-equation} gives $B=0$ and the energy drop is zero.
Otherwise, change all signs if necessary and assume $g>0$, so $B>0$.  Put
$H(t)=\sum_i\varphi_p(a_i-t)$.  Then $H$ is decreasing,
$H(0)=B$, and $H(g)=0$.  Summing \eqref{eq:midpoint-increment} with
$x=a_i$ and $h=g$ gives
\[
 H(g/2)=H(g/2)-H(g)\ge c_p^\ast[H(0)-H(g)]=c_p^\ast B.
\]
For the local objective $F(t)=\sum_i|a_i-t|^p$,
$F'(t)=-pH(t)$, and hence
\begin{equation}\label{eq:drop-integral-bounds}
 \frac{pc_p^\ast}{2}gB
 \le p\int_0^{g/2}H(t)\,dt
 \le F(0)-F(g)=p\int_0^gH(t)\,dt
 \le pgB.
\end{equation}
The same conclusion follows for $g<0$ after changing signs.  We have thus
proved, including all zero cases,
\begin{equation}\label{eq:drop-gB}
 \cD_p(f,u)\asymp_p |gB|.
\end{equation}

We shall also use the standard one-dimensional monotonicity inequality
\begin{equation}\label{eq:phi-monotonicity}
 (\varphi_p(x)-\varphi_p(y))(x-y)
 \asymp_p |x-y|^2(|x|+|y|)^{p-2}
\end{equation}
(with its continuous interpretation at the origin).  For completeness, this
follows directly from homogeneity and compactness.  Since
$\varphi_p'(s)=(p-1)|s|^{p-2}$ almost everywhere,
\[
 (\varphi_p(x)-\varphi_p(y))(x-y)
 =(p-1)|x-y|^2\int_0^1|y+t(x-y)|^{p-2}\,dt.
\]
If $S=|x|+|y|>0$, divide the two endpoints by $S$.  On the compact set
$\{(a,b):|a|+|b|=1\}$, the integral
$\int_0^1|b+t(a-b)|^{p-2}\,dt$ is finite, continuous, and bounded above and
below by positive constants depending only on $p$; here $p-2>-1$ makes a
possible crossing of zero integrable.  Rescaling by $S$ proves
\eqref{eq:phi-monotonicity}; the remaining case follows by continuity.

Suppose first that $p\ge2$.  At $p=2$, equation \eqref{eq:g-equation} gives
$B=dg$, so \eqref{eq:B-g-pge2} below holds with $A_f(u)=d$.  If $p>2$,
every difference
$\varphi_p(a_i)-\varphi_p(a_i-g)$ has the sign of $g$, summing the absolute
form of \eqref{eq:phi-monotonicity} gives
\[
 |B|\asymp_p |g|\sum_i(|a_i|+|a_i-g|)^{p-2}.
\]
The elementary triangle inequalities
\begin{equation}\label{eq:ag-comparison}
 \frac12(|a|+|g|)\le |a|+|a-g|\le2(|a|+|g|)
\end{equation}
and $(s+t)^{p-2}\asymp_p s^{p-2}+t^{p-2}$ now yield
\begin{equation}\label{eq:B-g-pge2}
 |B|\asymp_p |g|\bigl(A_f(u)+d|g|^{p-2}\bigr).
\end{equation}
The case $g=0$ was settled above, so the following ratio calculations concern
$g\ne0$.  To identify explicitly the minimum in
\eqref{eq:local-drop-comparison}, put
\[
 R=\frac{B^2}{A_f(u)},\qquad
 Q=\frac{|B|^q}{d^{1/(p-1)}}.
\]
When $p>2$ and $A_f(u)\ge d|g|^{p-2}$, equation
\eqref{eq:B-g-pge2} gives
\[
 \cD_p(f,u)\asymp_p |g|^2A_f(u)\asymp_p R,
 \qquad
 \frac QR\asymp_p
 \left(\frac{A_f(u)}{d|g|^{p-2}}\right)^{1/(p-1)}.
\]
Thus $R\lesssim_p Q$, so the minimum is $R$ up to a $p$-dependent
factor.  In the complementary regime $A_f(u)<d|g|^{p-2}$,
\[
 \cD_p(f,u)\asymp_p d|g|^p\asymp_p Q,
 \qquad
 \frac RQ\asymp_p\frac{d|g|^{p-2}}{A_f(u)}.
\]
Here $Q\lesssim_p R$, so the minimum is $Q$ up to a $p$-dependent factor.
At $p=2$, one has $A_f(u)=d$, $B=dg$, and $R=Q=d g^2$ exactly.
This proves the second line of \eqref{eq:local-drop-comparison}, including
two-sided comparability with the displayed minimum.

Now let $1<p<2$ and assume $g\ne0$, since the case $g=0$ was settled above.
The H\"older continuity of $\varphi_p$ and the first-order
estimate away from its singular point combine into the zero-safe scalar
bound
\begin{equation}\label{eq:phi-subquadratic-increment}
 |\varphi_p(a)-\varphi_p(a-g)|
 \le C_p\min\bigl\{|g|^{p-1},\ |g||a|^{p-2}\bigr\}.
\end{equation}
Here $0^{p-2}=+\infty$, so if $a=0$ the finite first term is selected.
For example, \eqref{eq:phi-subquadratic-increment} follows by considering
$|g|\le |a|/2$ and its complement, and using H\"older continuity in the
latter case.  Since the increments in \eqref{eq:g-equation} all have the
sign of $g$, summing each of the two bounds in
\eqref{eq:phi-subquadratic-increment} gives, separately,
\[
 |B|\le C_p d|g|^{p-1},
 \qquad |B|\le C_p|g|A_f(u).
\]
The second statement is automatic when $A_f(u)=\infty$.  Consequently,
with the convention finite divided by infinity equals zero,
\[
 |g|\ge c_p(|B|/d)^{1/(p-1)},
 \qquad |g|\ge c_p|B|/A_f(u).
\]
Multiplication by $|B|$ and \eqref{eq:drop-gB} prove the first line.
Equivalently, one may replace $|s|^p$ by
$(s^2+\delta^2)^{p/2}$ throughout and let $\delta\downarrow0$; this also
verifies the extended-real conventions at zero differences.
\end{proof}

Define the degree-weighted scalar Poincar\'e constant by
\begin{equation}\label{eq:Phi-p}
 \Phi_p(G)=\sup_{f\text{ nonconstant}}
 \frac{\displaystyle\inf_{c\in\R}
       \sum_{u\in V}\deg(u)|f(u)-c|^p}
      {\cE_p(f)}.
\end{equation}

\begin{thm}[Nonlinear Poincar\'e comparison]
\label{thm:nonlinear-poincare-comparison}
For every finite connected graph and every $1<p<\infty$,
\begin{equation}\label{eq:gap-Phi-comparison}
 \lambda_p(G)\ge c_p
 \begin{cases}
  \Phi_p(G)^{-1/(p-1)},&1<p<2,\\
  \Phi_p(G)^{-2/p},&p\ge2.
 \end{cases}
\end{equation}
\end{thm}

\begin{proof}
Translate $f$ so that a minimizer in the numerator of
\eqref{eq:Phi-p} is zero, and put
\[
 V_p(f)=\sum_u\deg(u)|f(u)|^p
 \le\Phi_p(G)\cE_p(f).
\]
Pairing the two orientations of every edge gives the exact identity
\begin{equation}\label{eq:fB-energy}
 \sum_uf(u)B_f(u)=\cE_p(f).
\end{equation}

If $1<p<2$, weighted H\"older's inequality in \eqref{eq:fB-energy} gives
\[
 \cE_p(f)
 \le\left(\sum_u
 \frac{|B_f(u)|^q}{\deg(u)^{1/(p-1)}}\right)^{1/q}
 V_p(f)^{1/p}.
\]
Lemma~\ref{lem:local-drop-comparison} therefore implies
\[
 \sum_u\cD_p(f,u)
 \ge c_p\Phi_p(G)^{-1/(p-1)}\cE_p(f).
\]

Suppose $p\ge2$.  Write
\[
 R_u=\frac{B_f(u)^2}{A_f(u)},\qquad
 Q_u=\frac{|B_f(u)|^q}{\deg(u)^{1/(p-1)}},
\]
and partition $V=V_1\sqcup V_2$ so that $R_u\le Q_u$ on $V_1$ and
$Q_u<R_u$ on $V_2$.  The signed partial sums
\[
 X=\sum_{u\in V_1}f(u)B_f(u),\qquad
 Y=\sum_{u\in V_2}f(u)B_f(u)
\]
satisfy $X+Y=\cE_p(f)$.  Hence either $X\ge\cE_p(f)/2$ or
$Y\ge\cE_p(f)/2$.

In the first case, Cauchy--Schwarz gives
\begin{equation}\label{eq:R-case}
 \sum_{u\in V_1}R_u
 \ge \frac{X^2}{\sum_uA_f(u)f(u)^2}.
\end{equation}
At $p=2$, the convention $A_f(u)=\deg(u)$ gives directly
\[
 \sum_uA_f(u)f(u)^2
 =\sum_u\deg(u)f(u)^2
 =V_2(f)
 \le \Phi_2(G)\cE_2(f).
\]
For $p>2$, H\"older's inequality, applied to the directed-edge sum in the
denominator, gives
\begin{align*}
 \sum_uA_f(u)f(u)^2
 &\le (2\cE_p(f))^{(p-2)/p}V_p(f)^{2/p}\\
 &\le C_p\Phi_p(G)^{2/p}\cE_p(f).
\end{align*}
Thus \eqref{eq:R-case} is at least
$c_p\Phi_p(G)^{-2/p}\cE_p(f)$.

In the second case, weighted H\"older's inequality, now restricted to $V_2$, gives
\[
 \sum_{u\in V_2}Q_u
 \ge c_p\Phi_p(G)^{-1/(p-1)}\cE_p(f).
\]
Finally, for any $c\in\R$,
\[
 \cE_p(f)\le2^{p-1}\sum_u\deg(u)|f(u)-c|^p,
\]
so $\Phi_p(G)\ge2^{1-p}$.  Since
$1/(p-1)\le2/p$, the last lower bound is, up to a $p$-dependent constant,
at least the required $\Phi_p(G)^{-2/p}$ bound.  Lemma
\ref{lem:local-drop-comparison} controls the relevant minimum in both cases.
Dividing by $\cE_p(f)$ and taking the infimum proves the theorem.
\end{proof}

For $1<p<\infty$ and $0<\varepsilon\le1$, write
\[
  \mathsf T_p(G,\varepsilon)
  :=\sup_{f_0\in[-1,1]^V}\mathbb E\tau_p(\varepsilon;f_0).
\]

\begin{prop}[Extrapolation from the quadratic relaxation gap]
\label{prop:comparison-with-p2}
For every finite connected graph and every $2\le p<\infty$,
\[
  \lambda_p(G)\ge c_p\lambda_2(G).
\]
Moreover, for $0<\varepsilon\le1$,
\begin{equation}
  \mathsf T_p(G,\varepsilon)
  \le C_p\frac{n}{\lambda_2(G)}
       \log\!\left(\frac{en}{\varepsilon}\right).
  \label{eq:Tp-vs-gap2}
\end{equation}
\end{prop}

\begin{proof}
Matou\v{s}ek's extrapolation inequality~\cite{Matousek1997} has a
reversible Markov-chain formulation; see Naor and Silberman
\cite[Lemma~4.4]{NaorSilberman2011}.  For a reversible transition matrix
$P$ with stationary measure $\pi$, define its pairwise Poincar\'e modulus by
\[
 \mathfrak P_r(P)=
 \sup_{f\text{ nonconstant}}
 \left(
 \frac{\displaystyle\sum_{u,v}\pi(u)\pi(v)|f(u)-f(v)|^r}
      {\displaystyle\sum_{u,v}\pi(u)P(u,v)|f(u)-f(v)|^r}
 \right)^{1/r}.
\]
To translate this normalization to \eqref{eq:Phi-p}, put
\[
  \pi(u)=\frac{\deg(u)}{2|E|},
  \qquad
  P(u,v)=\frac{\mathbf{1}_{\{u\sim v\}}}{\deg(u)}.
\]
This is a reversible Markov chain.  For every $1<r<\infty$, write
\[
 A_r=\sum_{u,v}\pi(u)P(u,v)|f(u)-f(v)|^r,
 \quad
 M_r=\inf_{c\in\R}\sum_u\pi(u)|f(u)-c|^r,
\]
and
\[
 B_r=\sum_{u,v}\pi(u)\pi(v)|f(u)-f(v)|^r.
\]
Then
\begin{align*}
 A_r
 &=\frac{\cE_r(f)}{|E|},\\
 M_r
 &=\frac{1}{2|E|}\inf_{c\in\R}
     \sum_u\deg(u)|f(u)-c|^r.
\end{align*}
Choosing $c=f(v)$ and averaging in $v$ proves $M_r\le B_r$; the triangle
inequality about a minimizer of $M_r$ gives $B_r\le2^rM_r$.  Consequently,
since $2M_r/A_r$ is the quotient in \eqref{eq:Phi-p}, taking suprema over
nonconstant $f$ yields
\[
 \frac{\Phi_r(G)}2\le \mathfrak P_r(P)^r
 \le 2^{r-1}\Phi_r(G).
\]
The reversible-chain extrapolation theorem states that
\[
 \mathfrak P_p(P)\le Cp\,\mathfrak P_2(P),
 \qquad 2\le p<\infty,
\]
where $C$ is universal.  Combining it with the preceding two-sided
normalization comparison gives
\begin{equation}
  \Phi_p(G)^{1/p}
  \le C p\,\Phi_2(G)^{1/2},
  \label{eq:poincare-extrapolation}
\end{equation}
where $C$ is universal.  Since
$\Phi_2(G)^{-1}=\lambda_2(G)$,
Theorem~\ref{thm:nonlinear-poincare-comparison} and
\eqref{eq:poincare-extrapolation} imply
\[
  \lambda_p(G)\ge c_p\lambda_2(G),
  \qquad 2\le p<\infty.
\]
Consequently, Theorem~\ref{thm:variational-upper} and the elementary bounds
$|E|\le n^2/2$ and $D\le n-1$ yield
\eqref{eq:Tp-vs-gap2}.
\end{proof}

\begin{cor}[Universal comparison with quadratic consensus]
\label{cor:universal-p-comparison}
For every finite connected graph, every $2\le p<\infty$, and every
$0<\varepsilon\le1/2$,
\begin{equation}
  \mathsf T_p(G,\varepsilon)
  \le C_p\log\!\left(\frac{en}{\varepsilon}\right)
          \mathsf T_2(G,\varepsilon).
  \label{eq:Tp-vs-T2}
\end{equation}
Thus the worst-profile $p$-consensus time on every graph is at most the
quadratic worst-profile time times one logarithm and a $p$-dependent
constant.
\end{cor}

\begin{proof}
For asynchronous DeGroot dynamics,
\cite[Claim~6.1]{elboim2024asynchronousdegrootdynamics} gives a spectral
lower bound under rate-one clocks at every vertex.  Converting it to the
present discrete update count gives
\[
  \mathsf T_2(G,1/2)\ge c\frac{n}{\lambda_2(G)}.
\]
Since $\mathsf T_2(G,\varepsilon)\ge\mathsf T_2(G,1/2)$ for
$0<\varepsilon\le1/2$, combining this bound with
\eqref{eq:Tp-vs-gap2} proves \eqref{eq:Tp-vs-T2}.
\end{proof}

Since $\mathsf T_p$ and $\mathsf T_2$ take their
suprema over initial profiles independently, \eqref{eq:Tp-vs-T2} does not
address the same-initial-profile comparison asked in
\cite[Question~6.1]{amir2025convergencerateellpenergyminimization}.

\subsection{The quadratic case}

Let $\mathbf L$ and $\mathbf D$ be the combinatorial Laplacian and degree
matrix.  Let $\mu_1(G)$ denote the first positive eigenvalue of the normalized
Laplacian $\mathbf D^{-1/2}\mathbf L\mathbf D^{-1/2}$.
The standard normalized-Laplacian spectral calculation gives
\[
 \lambda_2(G)=\mu_1(G);
\]
we use this familiar quadratic identity without a separate proof; see
Chung~\cite[Section~1.2]{chung1997spectral}.

\section{Lower-bound mechanisms}\label{sec:lower}

\subsection{A standard coupon obstruction}

\paragraph{Coupon observation.}
Partition $V=A\sqcup C$, where $a=|A|=\lfloor n/2\rfloor$ and
$b=|C|=\lceil n/2\rceil$, and put $f_0=\one_A$.  For nonempty $B\subseteq V$, let
\[
 \sigma_B=\min\{t\ge0:B\subseteq\{U_s:1\le s\le t\}\}
\]
be its cover time.  Before $\sigma_A\wedge\sigma_C$, untouched vertices in
$A$ and $C$ retain the values one and zero, respectively.  Hence
$\tau_p(\epsilon;f_0)\ge\sigma_A\wedge\sigma_C$ whenever $0<\epsilon<1$.

Writing $H_k=\sum_{j=1}^k j^{-1}$, the usual coupon-collector calculation
gives $\E\sigma_B=nH_{|B|}$: with $j$ vertices of $B$ still unseen, the wait
for a new one has mean $n/j$.  Since
$\sigma_A\vee\sigma_C=\sigma_V$, the identity
$x\wedge y+x\vee y=x+y$ gives
\[
 \E(\sigma_A\wedge\sigma_C)
 =n(H_a+H_b-H_n)\ge\frac n4\log n.
\]
Indeed,
$H_n-H_b=\sum_{j=1}^a(b+j)^{-1}\le H_a/2$ and
$H_a\ge\log(a+1)\ge\frac12\log n$.  Consequently,
\[
 \sup_{f_0\in\{0,1\}^V}\E\tau_p(\epsilon;f_0)
 \ge \frac n4\log n,
 \qquad p\in(1,\infty],\quad0<\epsilon<1.
\]

\subsection{Causal propagation}

The next result applies to every $p\in(1,\infty]$.  Its input is the size of
one synchronous update, not an energy or spectral estimate.

\begin{thm}[Propagation lower bound]\label{thm:propagation-lower}
Let $p\in(1,\infty]$ and let $f_0:V\to\R$ be nonconstant.  Set
\[
 R=\osc(f_0),\qquad
 \psi=\lVert Sf_0-f_0\rVert_\infty.
\]
Then $\psi>0$, and for every $0<\epsilon\le R/2$,
\begin{equation}\label{eq:propagation-lower}
 \E\tau_p(\epsilon;f_0)
 \ge \frac{1}{64e}\frac{nR}{\Delta\psi}.
\end{equation}
\end{thm}

\begin{proof}
If $Sf_0=f_0$, consider a vertex where $f_0$ is maximal.  A local $p$-mean
can equal the largest neighboring value only if all its inputs have that
value.  Connectivity then makes $f_0$ constant.  Hence $\psi>0$.

For $u\in V$, write
\[
 M_u(f)=m_p\bigl((f(v))_{v\sim u}\bigr).
\]
Lemma~\ref{lem:mean-properties} gives the stability estimate
\begin{equation}\label{eq:local-stability}
 |M_u(f)-M_u(g)|
 \le \max_{v\sim u}|f(v)-g(v)|.
\end{equation}

We now attach a nonnegative integer height to each vertex.  Initially,
$H_0(u)=0$.  At time $t$, if $u\ne U_t$, let
$H_t(u)=H_{t-1}(u)$; at the updated vertex set
\begin{equation}\label{eq:height-recursion}
 H_t(U_t)=1+\max_{v\sim U_t}H_{t-1}(v).
\end{equation}
We claim that
\begin{equation}\label{eq:height-controls-change}
 |f_t(u)-f_0(u)|\le\psi H_t(u)
 \qquad (u\in V,t\ge0).
\end{equation}
The claim holds at time zero.  If it holds at time $t-1$, then at the updated
vertex, \eqref{eq:local-stability} and the definition of $\psi$ give
\begin{align*}
 |f_t(U_t)-f_0(U_t)|
 &\le |M_{U_t}(f_{t-1})-M_{U_t}(f_0)|
      +|M_{U_t}(f_0)-f_0(U_t)|\\
 &\le \psi\left(1+\max_{v\sim U_t}H_{t-1}(v)\right).
\end{align*}
All other coordinates are unchanged, proving the induction step.

For every integer $k\ge1$, tracing \eqref{eq:height-recursion} backwards
shows that $H_t(u)\ge k$ requires a nearest-neighbor sequence
$x_1,\ldots,x_k=u$ and times
$1\le s_1<\cdots<s_k\le t$ with $U_{s_i}=x_i$.  There are at most
$\Delta^{k-1}$ such vertex sequences.  Since each specified list of labels
has probability $n^{-k}$, a union bound and
$\binom{t}{k}\le(et/k)^k$ give
\begin{equation}\label{eq:height-tail}
 \Pr(H_t(u)\ge k)
 \le \Delta^{k-1}\binom{t}{k}n^{-k}
 \le \Delta^{-1}\left(\frac{e\Delta t}{nk}\right)^k
 \le \left(\frac{e\Delta t}{nk}\right)^k,
\end{equation}
where $\Delta\ge1$ because $G$ is connected and $n\ge2$.

Choose vertices $a,b$ with $f_0(a)-f_0(b)=R$, and put
\[
 k=\left\lfloor\frac{R}{8\psi}\right\rfloor+1,
 \qquad
 t=\left\lfloor\frac{nk}{4e\Delta}\right\rfloor.
\]
By \eqref{eq:height-tail}, each of
$\Pr(H_t(a)\ge k)$ and $\Pr(H_t(b)\ge k)$ is at most $4^{-k}$.
Thus, with probability at least $1/2$, both heights are smaller than $k$.
On that event, \eqref{eq:height-controls-change} gives
\[
 f_t(a)-f_t(b)
 \ge R-2\psi(k-1)
 \ge \frac{3R}{4}.
\]
In particular, $\osc(f_t)>\epsilon$ whenever $\epsilon\le R/2$.
Oscillation is nonincreasing by Lemma~\ref{lem:oscillation-energy}, so on
this event the process cannot have crossed the threshold earlier; hence
$\tau_p(\epsilon;f_0)>t$.  Since $t+1>nk/(4e\Delta)$,
\[
 \E\tau_p(\epsilon;f_0)
 \ge (t+1)\Pr(\tau_p(\epsilon;f_0)>t)
 >\frac{nk}{8e\Delta}
 \ge\frac{nR}{64e\Delta\psi},
\]
as required.
\end{proof}

\subsection{The Lipschitz-learning endpoint}

For $p=\infty$, the local rule \eqref{eq:infinity-mean} replaces a selected
value by the midrange of its neighbors.  This is the asynchronous
Lipschitz-learning dynamics.  We first construct the slowly changing profile
used for its diameter lower bound.

The next lemma is the two-point-boundary specialization of the finite-graph
minimal Lipschitz-extension construction in
\cite[Proposition~3.1]{amir2025convergencerateellpenergyminimization}.

\begin{lem}[Minimal Lipschitz extension]
\label{lem:minimal-lipschitz-extension}
Let $s,t\in V$ satisfy $\operatorname{dist}(s,t)=D=\diam(G)$.  There exists
$g:V\to[-1,1]$ such that
\begin{enumerate}[label=(\alph*)]
\item $g(s)=-1$ and $g(t)=1$;
\item $|g(u)-g(v)|\le2/D$ for every edge $\{u,v\}$; and
\item for every $u\notin\{s,t\}$,
\begin{equation}\label{eq:infinity-harmonic}
 g(u)=\frac{\max_{v\sim u}g(v)+\min_{v\sim u}g(v)}2.
\end{equation}
\end{enumerate}
\end{lem}

\begin{proof}
Among all extensions $h:V\to[-1,1]$ with $h(s)=-1$ and $h(t)=1$, list the
numbers $|h(u)-h(v)|$, one for each edge, in nonincreasing order.  Choose an
extension $g$ for which this finite vector is lexicographically minimal.
Existence follows by minimizing its coordinates successively over a compact
set.

Fix $u\notin\{s,t\}$, and let
$a=\min_{v\sim u}g(v)$, $b=\max_{v\sim u}g(v)$, and $c=g(u)$.  If $a=b$, then
lexicographic minimality forces $c=a$: when $c>a$, replacing $c$ by
$c-\delta$ for any sufficiently small $0<\delta<c-a$ strictly decreases every
incident edge difference, and when $c<a$ the symmetric upward perturbation
does the same.  Thus assume $a<b$.  Suppose first that $c>(a+b)/2$, and put
\[
 M_-=c-a,
 \qquad M_+=|b-c|.
\]
Then $M_->M_+$.  Choose
\[
 0<\delta<\frac12(M_--M_+)
\]
small enough that $c-\delta\in[-1,1]$, and replace $g(u)$ by
$c-\delta$.  Every new incident edge difference is at most
\[
 \max\{M_--\delta,M_++\delta,\delta\}<M_-.
\]
Before the perturbation, at least one incident edge difference was exactly
$M_-$, while all nonincident edge differences are unchanged.  Thus, even if
other incident or nonincident edges tie at $M_-$, the number of entries equal
to $M_-$ in the globally sorted vector strictly decreases, no larger entry
changes, and no new entry reaches $M_-$.  The sorted vector has therefore
decreased lexicographically, a contradiction.  If $c<(a+b)/2$, the symmetric
perturbation increasing $c$ gives the same contradiction, with
$M_+=b-c>M_-=|c-a|$.  Hence $c=(a+b)/2$, proving
\eqref{eq:infinity-harmonic}.

The comparison extension
\[
 q(v)=\max\left\{-1,\min\left\{1,-1+
 \frac{2\operatorname{dist}(s,v)}D\right\}\right\}
\]
has the prescribed boundary values and changes by at most $2/D$ across an
edge.  The first coordinate of the lexicographic vector for $g$ is no larger
than that for $q$, proving (b).
\end{proof}

\begin{thm}[Diameter lower bound for Lipschitz learning]
\label{thm:lipschitz-diameter-lower}
Let $G$ have $n$ vertices, maximum degree $\Delta$, and diameter $D$.
There exists $f_0\in[-1,1]^V$ such that, for every
$0<\epsilon\le1$,
\begin{equation}\label{eq:lipschitz-diameter-lower}
 \E\tau_\infty(\epsilon;f_0)
 \ge \frac{1}{16e\pi^2}\frac{nD^2}{\Delta}.
\end{equation}
\end{thm}

\begin{proof}
Choose $s,t$ and $g$ as in
Lemma~\ref{lem:minimal-lipschitz-extension}, and define
\[
 f_0(v)=\sin\left(\frac{\pi g(v)}2\right).
\]
Then $f_0(s)=-1$, $f_0(t)=1$, and $\osc(f_0)=2$.

Fix $u\notin\{s,t\}$ and put
$a=\min_{v\sim u}g(v)$ and $b=\max_{v\sim u}g(v)$.  Since sine is
increasing on $[-\pi/2,\pi/2]$, the synchronous midrange update and
\eqref{eq:infinity-harmonic} give
\begin{align*}
 (Sf_0)(u)
 &=\frac12\left[
 \sin\left(\frac{\pi a}2\right)
 +\sin\left(\frac{\pi b}2\right)\right]\\
 &=f_0(u)\cos\left(\frac{\pi(b-a)}4\right).
\end{align*}
Property (b) of the lemma gives $b-a\le4/D$.  Using
$1-\cos x\le x^2/2$, we obtain
\begin{equation}\label{eq:interior-lipschitz-change}
 |(Sf_0)(u)-f_0(u)|\le\frac{\pi^2}{2D^2}.
\end{equation}

Every neighbor $v$ of $s$ has
$-1\le g(v)\le-1+2/D$.  Therefore
\[
 0\le f_0(v)+1
 =1-\cos\left(\frac{\pi(g(v)+1)}2\right)
 \le\frac{\pi^2}{2D^2}.
\]
The midrange of these neighboring values obeys the same bound, and the
argument at $t$ is symmetric.  Together with
\eqref{eq:interior-lipschitz-change}, this proves
\[
 \lVert Sf_0-f_0\rVert_\infty\le\frac{\pi^2}{2D^2}.
\]
Theorem~\ref{thm:propagation-lower}, with $R=2$, now gives
\eqref{eq:lipschitz-diameter-lower}.
\end{proof}

\section{Boxes}\label{sec:boxes}

\subsection{A tensor Poincar\'e inequality}

Write $[L]^d=\{0,\ldots,L-1\}^d$ and give it nearest-neighbor edges.

\begin{lem}[Tensor Poincar\'e inequality]\label{lem:box-poincare}
For every fixed $1<p<\infty$, all integers $d\ge1$ and $L\ge2$, and every
$f:[L]^d\to\R$,
\begin{equation}\label{eq:box-unweighted-poincare}
 \sum_{x\in[L]^d}|f(x)-\bar f|^p
 \le C_p d^{\max\{0,p/2-1\}}L^p\cE_p(f),
 \qquad \bar f=L^{-d}\sum_xf(x).
\end{equation}
Consequently,
\begin{equation}\label{eq:box-Phi}
 \Phi_p([L]^d)\le C_p
 \begin{cases}
  dL^p,&1<p<2,\\
  d^{p/2}L^p,&p\ge2.
 \end{cases}
\end{equation}
\end{lem}

\begin{proof}
Let $X=(X_1,\ldots,X_d)$ be uniform on $[L]^d$, expose its coordinates in
order, and set
\[
 M_i=\E_X[f(X)\mid X_1,\ldots,X_i],\qquad Z_i=M_i-M_{i-1}.
\]
Burkholder's martingale square-function inequality
\cite{burkholder1973distribution} gives
\[
 \E_X|f(X)-\bar f|^p
 \le C_p\E_X\left(\sum_{i=1}^dZ_i^2\right)^{p/2}.
\]
For $p<2$, the last power is at most $\sum_i|Z_i|^p$; for $p\ge2$, it is
at most $d^{p/2-1}\sum_i|Z_i|^p$.

We record the one-dimensional estimate used here.  If
$a_0,\ldots,a_{L-1}\in\R$ and $\bar a=L^{-1}\sum_i a_i$, then Jensen's
inequality, telescoping, and H\"older's inequality give
\begin{align*}
 \sum_{i=0}^{L-1}|a_i-\bar a|^p
 &\le \frac1L\sum_{i,j=0}^{L-1}|a_i-a_j|^p\\
 &\le L^p\sum_{k=0}^{L-2}|a_{k+1}-a_k|^p.
\end{align*}
Indeed, for $i<j$ one writes
$a_i-a_j=\sum_{k=i}^{j-1}(a_k-a_{k+1})$ and applies H\"older with
$j-i\le L$.  Apply this path inequality to $M_i$ as a function of $X_i$,
conditional on the preceding coordinates.  Jensen's inequality over the
unexposed coordinates then gives
\[
 \E_X|Z_i|^p\le C_p L^{p-d}
 \sum_{\substack{\{x,y\}\in E\\x-y=\pm e_i}}
 |f(x)-f(y)|^p.
\]
Summing in $i$ and multiplying by $L^d$ proves
\eqref{eq:box-unweighted-poincare}.

For any $f$, the minimizing degree-weighted $p$-center is no worse than the
choice $c=\bar f$, and $\deg(x)\le2d$.  Hence
\[
 \inf_c\sum_x\deg(x)|f(x)-c|^p
 \le2d\sum_x|f(x)-\bar f|^p.
\]
Together with \eqref{eq:box-unweighted-poincare}, this is
\eqref{eq:box-Phi}.
\end{proof}

\subsection{Slow profiles on boxes}

The next elementary estimate records the two one-dimensional profiles used
below.  Put $N=L-1$ and, for $r\ge2$, define
\begin{equation}\label{eq:graded-profile}
 H_r(t)=
 \begin{cases}
  -1+(2t)^r,&0\le t\le1/2,\\
  1-[2(1-t)]^r,&1/2\le t\le1.
 \end{cases}
\end{equation}

\begin{lem}\label{lem:graded-differences}
Let $h_i=H_r(i/N)$, with the evident interpretation when $N=1$.
\begin{enumerate}[label=(\alph*)]
\item If $1<p<2$, $q=p/(p-1)$, and $r=q$, then
\[
 \left|\sum_{j:\,|j-i|=1}
 |h_i-h_j|^{p-2}(h_i-h_j)\right|\le C_pL^{-p}
 \quad(0\le i\le N).
\]
\item If $r=2$, then the second differences at interior indices and the
one-sided first differences at the two endpoints are bounded by $C L^{-2}$.
\end{enumerate}
\end{lem}

\begin{proof}
For part (a), put $a=p-1$ and
$\delta_i=h_{i+1}-h_i$ for $0\le i<N$.  By the mean-value theorem, for some
$\xi_i\in(i/N,(i+1)/N)$,
\[
 \delta_i=N^{-1}H_q'(\xi_i),\qquad
 \delta_i^a=N^{-a}\bigl(H_q'(\xi_i)\bigr)^a.
\]
Here
\[
 q=\frac{p}{p-1}=1+\frac1a,
\]
and hence
\[
 F(t):=\bigl(H_q'(t)\bigr)^a
      =(q\,2^q)^a\min\{t,1-t\}.
\]
Thus $F$ is globally Lipschitz on $[0,1]$, including across $1/2$, and
$F(0)=F(1)=0$.  For $1\le i\le N-1$,
\begin{align*}
 \left|\sum_{j:\,|j-i|=1}
 |h_i-h_j|^{p-2}(h_i-h_j)\right|
 &=|\delta_{i-1}^a-\delta_i^a|\\
 &\le C_pN^{-a}|\xi_{i-1}-\xi_i|
 \le C_pN^{-p}.
\end{align*}
At the endpoints, $F(\xi_0)\le C_p/N$ and
$F(\xi_{N-1})\le C_p/N$, so the corresponding one-sided terms are also
bounded by $C_pN^{-p}$.

For part (b), redefine $\delta_i=h_{i+1}-h_i$ with $h_i=H_2(i/N)$.
Since
\[
 H_2'(t)=8\min\{t,1-t\},
\]
the derivative $H_2'$ is globally Lipschitz on $[0,1]$ and vanishes at both
endpoints.  Applying the mean-value theorem on consecutive mesh intervals
gives
\[
 |h_{i+1}-2h_i+h_{i-1}|=|\delta_i-\delta_{i-1}|\le CN^{-2}
 \qquad(1\le i\le N-1),
\]
and the two endpoint increments are bounded by $CN^{-2}$ as well.  Since
$N=L-1\ge L/2$, both conclusions have the stated bounds.
\end{proof}

\begin{thm}[Consensus bounds on a box]\label{thm:box-bounds}
Fix $1<p<\infty$, put $q=p/(p-1)$, and let $d,L\ge2$ be integers,
$n=L^d$, and
$0<\epsilon\le1$.  For every
$f_0\in[-1,1]^{[L]^d}$,
\begin{equation}\label{eq:box-upper}
 \E\tau_p(\epsilon;f_0)\le C_p
 \begin{cases}
 n d^{1/(p-1)}L^{p/(p-1)}\log(n/\epsilon),&1<p<2,\\
 n dL^2\log(n/\epsilon),&p\ge2.
 \end{cases}
\end{equation}
For the worst initial profile,
\begin{equation}\label{eq:box-lower}
 \sup_{f_0\in[-1,1]^{[L]^d}}\E\tau_p(1/2;f_0)\ge c_pn
 \begin{cases}
 \max\{d^{1/(p-1)-1}L^q,d\log L\},&1<p<2,\\
 dL^2,&p=2,\\
 \max\{d\log L,L^2\},&p>2.
 \end{cases}
\end{equation}
The terms inside a maximum may be realized by different profiles.
\end{thm}

\begin{proof}
Theorem~\ref{thm:nonlinear-poincare-comparison} and
Lemma~\ref{lem:box-poincare} give
\[
 \lambda_p([L]^d)\ge c_p
 \begin{cases}
 d^{-1/(p-1)}L^{-p/(p-1)},&1<p<2,\\
 d^{-1}L^{-2},&p\ge2.
 \end{cases}
\]
The diameter and the number of edges are polynomial in $n$, so
Theorem~\ref{thm:variational-upper} gives \eqref{eq:box-upper}.

For $1<p<2$, put $q=p/(p-1)$ and take
$f_0(x)=H_q(x_1/N)$.  Fix a vertex and let $r_x$ be the number of its
neighbors in the inactive coordinates; $r_x\ge d-1$.  Write
\[
 g=(Sf_0)(x)-f_0(x),
 \qquad b_v=f_0(v)-f_0(x).
\]
The first-order equation for the local $p$-mean is
\[
 0=r_x|g|^{p-2}g+
 \sum_{\substack{v\sim x\\v\text{ active}}}
 |g-b_v|^{p-2}(g-b_v).
\]
Denote the sum by $F_x(g)$.  The function $F_x$ is increasing, and
$F_x(0)=B_{f_0}(x)$.  Hence $gF_x(0)\le0$; in particular, if
$F_x(0)=0$, then $g=0$.
The first-order equation and monotonicity of $F_x$ then give
\[
 r_x|g|^{p-1}=|F_x(g)|\le |F_x(0)|\le C_pL^{-p},
\]
where the last inequality is Lemma~\ref{lem:graded-differences}(a).
Since $d\ge2$ and $r_x\ge d-1\ge d/2$, we obtain
\[
 |g|\le C_p d^{-1/(p-1)}L^{-p/(p-1)}.
\]
Thus
\begin{equation}\label{eq:box-psi-subquadratic}
 \lVert Sf_0-f_0\rVert_\infty
 \le C_p d^{-1/(p-1)}L^{-p/(p-1)}.
\end{equation}

For $p>2$, take instead
\begin{equation}\label{eq:box-quadratic-profile}
 f_0(x)=\frac1d\sum_{j=1}^dH_2(x_j/N).
\end{equation}
For every interior coordinate $j$, define the signed neighbor increments
\begin{align*}
 \delta_j^\pm&=f_0(x\pm e_j)-f_0(x),\\
 c_j&=\frac{\delta_j^++\delta_j^-}{2},\\
 s_j&=\frac{\delta_j^+-\delta_j^-}{2}.
\end{align*}
The two increments are $c_j\pm s_j$.  The second-difference estimate in
Lemma~\ref{lem:graded-differences}(b) gives
\[
 |c_j|\le \frac{C}{dL^2}.
\]
At a boundary coordinate, the endpoint estimate in the same lemma gives
the bound $C/(dL^2)$ for the absolute value of the sole signed neighbor
increment.
Replace every interior pair according to
\[
 \{c_j+s_j,c_j-s_j\}\longmapsto\{s_j,-s_j\},
\]
and replace every boundary singleton by zero.  The resulting
multiset of increments is symmetric about zero, so its $p$-mean is zero.
Equivalently, the associated neighbor values have $p$-mean $f_0(x)$.
Every entry changed by at most $C/(dL^2)$; the sup-norm nonexpansiveness
in Lemma~\ref{lem:mean-properties} yields
\begin{equation}\label{eq:box-psi-superquadratic}
 \lVert Sf_0-f_0\rVert_\infty\le\frac{C}{dL^2}.
\end{equation}

Both nonlinear profiles have range two and the box has maximum degree at
most $2d$.
Theorem~\ref{thm:propagation-lower}, applied to
\eqref{eq:box-psi-subquadratic} and to
\eqref{eq:box-psi-superquadratic} when $p>2$, gives respectively the terms
$nd^{1/(p-1)-1}L^q$ and $nL^2$ in \eqref{eq:box-lower}.

 At $p=2$, let $P$ be the simple-random-walk matrix and let $\gamma$ be its
 spectral gap.  Claim~6.1 of Elboim et
 al.~\cite{elboim2024asynchronousdegrootdynamics} gives, in their
 continuous-time convention with one rate-one clock per vertex, a profile
 with $\E\tau_2(1/2)\ge(10\gamma)^{-1}$.  The embedded discrete chain has
 total clock rate $n$, so its expected update count is $n$ times the expected
 continuous time.  To identify $\gamma$, first consider the product chain
 that chooses one of the $d$ coordinates uniformly and, in that coordinate,
 moves left or right with probability $1/2$, replacing a missing endpoint
 move by holding.  Its stationary distribution is uniform.  The
 one-dimensional reflected path chain has eigenvalues $\cos(k\pi/L)$,
 $0\le k\le L-1$.  Since the product kernel is the average of the $d$
 coordinate kernels, its spectral gap is
 \[
  \gamma_{\rm prod}=\frac{1-\cos(\pi/L)}d\asymp\frac1{dL^2}.
 \]
 For the simple
 random walk $P$, every box degree lies between $d$ and $2d$.  Moreover, its
 stationary weights are proportional to degree and its edge conductances are
 constant, so both are comparable, by absolute factors, with the uniform
 stationary weights and edge conductances of the product chain.  The standard
 Dirichlet-form comparison therefore preserves the gap up to constants.
 These path, product, and comparison facts are in Levin and
 Peres~\cite[Sections~12.3, 12.4, and 13.3]{levin2017markov}; hence
 $\gamma([L]^d)\asymp(dL^2)^{-1}$.
 Thus Claim~6.1 gives the middle line of
\eqref{eq:box-lower}.  Finally, the coupon observation in
Section~\ref{sec:lower} gives $cn\log n=cnd\log L$.  It supplies the other
term both below and above two, and is dominated by $ndL^2$ at $p=2$.
Since the causal and coupon arguments use different initial profiles, their
conclusions combine after taking the supremum.
\end{proof}

\begin{cor}[Polylogarithmic comparison on boxes]
\label{cor:box-polylogarithmic-comparison}
Under the hypotheses of Theorem~\ref{thm:box-bounds}, define
\[
 U_p(d,L)=
 \begin{cases}
  nd^{1/(p-1)}L^{p/(p-1)},&1<p<2,\\
  ndL^2,&p\ge2.
 \end{cases}
\]
Then
\begin{equation}\label{eq:box-polylogarithmic-lower}
 \sup_{f_0\in[-1,1]^{[L]^d}}\E\tau_p(1/2;f_0)
 \ge c_p
 \begin{cases}
  U_p(d,L)/\log(en),&p\ne2,\\
  U_2(d,L),&p=2.
 \end{cases}
\end{equation}
Consequently, at fixed tolerance $1/2$, the upper and lower bounds in
Theorem~\ref{thm:box-bounds} differ by at most
$O_p((\log n)^2)$ for $p\ne2$ and by at most $O(\log n)$ for $p=2$.
The causal-profile term by itself leaves a factor $d$.  Once the coupon term
is included, however, the pointwise discrepancy can be smaller; see
Remark~\ref{rem:box-dimension-gap}.
\end{cor}

\begin{proof}
Because $n=L^d$ and $L\ge2$,
\[
 d\le\frac{\log n}{\log2}\le C\log(en).
\]
For $1<p<2$, the first term in \eqref{eq:box-lower} is
\[
 nd^{1/(p-1)-1}L^{p/(p-1)}=\frac{U_p(d,L)}d.
\]
For $p>2$, the $nL^2$ term in \eqref{eq:box-lower} equals
$U_p(d,L)/d$.  These identities and the preceding bound on $d$ prove the
first line of \eqref{eq:box-polylogarithmic-lower}.  At $p=2$, the middle
line of \eqref{eq:box-lower} is exactly $U_2(d,L)$.  Finally,
\eqref{eq:box-upper} at tolerance $1/2$ is at most
$C_pU_p(d,L)\log(en)$, which proves the comparison factors.
\end{proof}

\begin{rem}\label{rem:box-dimension-gap}
At tolerance $1/2$, after suppressing the multiplicative logarithm in
\eqref{eq:box-upper}, the ratio of its displayed scale $U_p(d,L)$ to the
detailed lower-bound scale in \eqref{eq:box-lower} is
\[
 \begin{cases}
  \displaystyle
  \min\left\{d,
  \frac{d^{(2-p)/(p-1)}L^{p/(p-1)}}{\log L}\right\},&1<p<2,\\[8pt]
  1,&p=2,\\[4pt]
  \displaystyle
  \min\left\{d,\frac{L^2}{\log L}\right\},&p>2.
 \end{cases}
\]
Thus the causal-profile term alone leaves a factor $d$, whereas the combined
lower bound leaves a pointwise ratio no larger than $d$ and sometimes
smaller.  In particular, for $p>2$ and fixed $L$, the coupon term gives the
same linear dependence on $d$ as the upper bound.  The two lower-bound
mechanisms do not yet convert the sharp local displacement
\eqref{eq:box-psi-subquadratic} or \eqref{eq:box-psi-superquadratic} into a
degree-free nonlinear influence estimate for all parameter ranges.
\end{rem}

\section{Trees}\label{sec:trees}

Throughout this section, fix $1<p<\infty$.  Let $G$ be a nontrivial finite
tree, and let $\Delta$ denote its actual maximum graph degree.  Root $G$ at
$r$.  A vertex with no
children is a leaf, and the weights $w_r(v)$ are defined recursively by
\begin{equation}\label{eq:tree-weight}
 w_r(v)=1\quad\hbox{at a leaf},\qquad
 w_r(v)=\left(\sum_{z\text{ child of }v}
 (w_r(z)+1)^{p-1}\right)^{1/(p-1)}
\end{equation}
at a nonleaf.  Put
\begin{equation}\label{eq:T-root}
 T_G(r)=\max_x\sum_{\substack{v\in\operatorname{Path}(x,r)\\v\ne r}}
 w_r(v),\qquad T_G=\min_{r\in V}T_G(r).
\end{equation}
Both $w_r$ and $T_G$ depend on the exponent fixed above; in particular,
$T_G=T_G^{(p)}$.  More explicit notation for the weights would be $w_{r,p}$,
but we suppress the index $p$ throughout this section.  Although rerooting
changes the individual weights
and the one-root quantity $T_G(r)$, the minimum in \eqref{eq:T-root} makes
$T_G$ a root-optimized, rerooting-invariant parameter of the unrooted tree,
independent of a distinguished drawing or orientation.

The bounds below should be read with their degree dependence visible.  At
constant accuracy, on bounded-degree trees $T_G$ controls the worst-profile
consensus time up to the displayed logarithmic factors.  On an
arbitrary-degree tree the available lower bound loses a factor of the actual
maximum degree $\Delta$; since $\Delta$ may grow polynomially with $n$, the
results do not give a degree-free characterization for all trees.

\subsection{Scalar residual inequalities}

We first isolate the local algebra.  Set $a=p-1$ and $q=p/a$.  The variables
below are magnitudes; their signs enter only through the scalar residual.

\begin{lem}[Power residual inequality]\label{lem:tree-power-residual}
Let $x_0,x_1,\ldots,x_k\ge0$, let $\sigma_i\in\{-1,1\}$, and put
\[
 \rho=\left|\sum_{i=0}^k\sigma_i x_i^a\right|.
\]
If $\alpha_i>0$ and $\sum_{i=1}^k\alpha_i=1$, then, for
$0<\varepsilon\le1$,
\begin{equation}\label{eq:tree-power-residual}
 x_0^p\le(1+\varepsilon)
 \sum_{i=1}^k\alpha_i^{-1/a}x_i^p
 +C_p\varepsilon^{-1/a}\rho^q.
\end{equation}
\end{lem}

\begin{proof}
Put $A=x_0^a$ and $b=\sum_{i=1}^kx_i^a$.  Weighted H\"older, or equivalently
the minimization of $\sum_i\alpha_i^{-1/a}y_i^q$ subject to
$\sum_i y_i=b$, gives
\begin{equation}\label{eq:weighted-child-power}
 b^q\le\sum_{i=1}^k\alpha_i^{-1/a}x_i^p.
\end{equation}
If $A\le b$, this proves the result.  If $A>b$, the triangle inequality gives
$\rho\ge A-b$.  The scalar Young inequality
\[
 (b+s)^q\le(1+\varepsilon)b^q
 +C_q\varepsilon^{-(q-1)}s^q
\]
with $s=A-b$ proves \eqref{eq:tree-power-residual}, because
$q-1=1/a$.
\end{proof}

For superquadratic powers, normalizing by the recursive tree weights gives a
degree-free local estimate.

\begin{lem}[Normalized superquadratic residual]
\label{lem:tree-normalized-superquadratic}
Let $p>2$, put $a=p-1$, and let $w_1,\ldots,w_k\ge1$, where $k\ge1$.  Set
\[
 w=\left(\sum_{i=1}^k(w_i+1)^a\right)^{1/a}.
\]
For signed numbers $a_0,a_1,\ldots,a_k$, define
\[
 \mathfrak D_p(a)=\sum_{j=0}^k|a_j|^p
 -\min_{t\in\R}\sum_{j=0}^k|a_j-t|^p.
\]
Then
\begin{equation}\label{eq:tree-normalized-superquadratic}
 \frac{|a_0|^p}{w}
 \le \sum_{i=1}^k\frac{|a_i|^p}{w_i}
 +C_p\mathfrak D_p(a).
\end{equation}
\end{lem}

\begin{proof}
Let $g$ be the minimizer in the definition of $\mathfrak D_p(a)$ and put
$c_j=a_j-g$.  With $\varphi_p(s)=|s|^{p-2}s$, stationarity gives
\begin{equation}\label{eq:normalized-stationarity}
 \sum_{j=0}^k\varphi_p(c_j)=0.
\end{equation}
Write $y_j=\varphi_p(c_j)$ and $q=p/(p-1)$.  Define
\[
 H(t)=\frac{|c_0+t|^p}{w}
 -\sum_{i=1}^k\frac{|c_i+t|^p}{w_i}.
\]
Then $H(g)$ is the left side of \eqref{eq:tree-normalized-superquadratic}
minus its propagated child term.

By \eqref{eq:normalized-stationarity}, $y_0=-\sum_i y_i$.  Weighted H\"older,
applied with the weights $w_i+1$, gives
\begin{equation}\label{eq:normalized-holder}
 |y_0|^q\le w\sum_{i=1}^k\frac{|y_i|^q}{w_i+1}.
\end{equation}
Consequently, if $h=-H(0)$, then
\begin{equation}\label{eq:normalized-slack}
 h\ge\sum_{i=1}^k
 \frac{|y_i|^q}{w_i(w_i+1)}\ge0.
\end{equation}
This is the useful slack between the weights $w_i+1$ in
\eqref{eq:normalized-holder} and the weights $w_i$ in the conclusion.

Put
\[
 S_0=\sum_{j=0}^k|y_j|^{2-q}
 =\sum_{j=0}^k|c_j|^{p-2}.
\]
Since
\[
 \frac{H'(0)}p=\frac{y_0}{w}-\sum_{i=1}^k\frac{y_i}{w_i},
\]
Cauchy--Schwarz, after factoring each $y_j$ into powers
$|y_j|^{(2-q)/2}$ and $|y_j|^{q/2}$ with its sign, yields
\begin{equation}\label{eq:normalized-derivative}
 \frac{|H'(0)|^2}{p^2}
 \le S_0\left(
 \frac{|y_0|^q}{w^2}+
 \sum_{i=1}^k\frac{|y_i|^q}{w_i^2}
 \right).
\end{equation}
The child sum in parentheses is at most $2h$ by
\eqref{eq:normalized-slack}.  Also, $w\ge w_i+1$ and
\eqref{eq:normalized-holder} imply
\[
 \frac{|y_0|^q}{w^2}
 \le\sum_{i=1}^k\frac{|y_i|^q}{w(w_i+1)}
 \le\sum_{i=1}^k\frac{|y_i|^q}{w_i(w_i+1)}
 \le h.
\]
Thus $|H'(0)|^2\le3p^2hS_0$, and completing the square gives
\begin{equation}\label{eq:normalized-linear-bound}
 H(0)+gH'(0)
 \le-h+\sqrt3p|g|\sqrt{hS_0}
 \le\frac{3p^2}{4}g^2S_0.
\end{equation}

For $s,t\in\R$, let
\[
 R_p(s,t)=|s+t|^p-|s|^p-p\varphi_p(s)t.
\]
The integral Taylor formula
\[
 R_p(s,t)=p(p-1)t^2\int_0^1(1-\theta)
 |s+\theta t|^{p-2}\,d\theta
\]
and an elementary split according to $|t|\le |s|/2$ or
$|t|>|s|/2$ show that
\begin{equation}\label{eq:normalized-bregman}
 c_pt^2(|s|+|t|)^{p-2}
 \le R_p(s,t)
 \le C_pt^2(|s|+|t|)^{p-2}.
\end{equation}
This includes $s=0$, $t=0$, and intervals crossing the origin.  By
\eqref{eq:normalized-stationarity},
\begin{equation}\label{eq:normalized-drop-bregman}
 \mathfrak D_p(a)=\sum_{j=0}^kR_p(c_j,g)
 \ge c_pg^2\sum_{j=0}^k(|c_j|+|g|)^{p-2}.
\end{equation}
Every Bregman remainder is nonnegative, so
\begin{align*}
 H(g)-H(0)-gH'(0)
 &=\frac{R_p(c_0,g)}w-
 \sum_{i=1}^k\frac{R_p(c_i,g)}{w_i}\\
 &\le C_pg^2(|c_0|+|g|)^{p-2}.
\end{align*}
Combining this estimate with \eqref{eq:normalized-linear-bound}, using
$S_0\le\sum_j(|c_j|+|g|)^{p-2}$, and then applying
\eqref{eq:normalized-drop-bregman} proves $H(g)\le C_p\mathfrak D_p(a)$.
If $g=0$, the same conclusion follows directly from $H(0)=-h\le0$.
All zero cases are covered because the exponents above are nonnegative.
\end{proof}

For subquadratic powers a unary vertex needs a sharper dependence on
$\varepsilon$.

\begin{lem}[Two-edge subquadratic residual]
\label{lem:tree-two-edge-residual}
Let $1<p<2$, let $x_0,x_1\ge0$, let
$\sigma_0,\sigma_1\in\{-1,1\}$, and let
$\rho=|\sigma_0x_0^{p-1}+\sigma_1x_1^{p-1}|$.  Then
\begin{equation}\label{eq:tree-two-edge-residual}
 x_0^p\le(1+\varepsilon)x_1^p+C_p\varepsilon^{-1}
 \max\left\{\rho^q,
 \frac{\rho^2}{x_0^{p-2}+x_1^{p-2}}\right\},
 \qquad0<\varepsilon\le1,
\end{equation}
where, if $x_i=0$, we set $x_i^{p-2}=+\infty$ and interpret a finite
nonnegative number divided by $+\infty$ as zero.
\end{lem}

\begin{proof}
Put $A=x_0^{p-1}$ and $b=x_1^{p-1}$; now $q>2$.  If the two signed terms
have the same sign, then $\rho\ge A$ and the first residual term controls
$x_0^p=A^q$.  For opposite signs, $A\le b$ implies $A^q\le b^q$, so the
leading term on the right already suffices.  It remains to consider $A>b$,
in which case $\rho=A-b$.  If $A\ge2b$, then $\rho\ge A/2$, and again
$\rho^q$ suffices.  If $b<A<2b$, Taylor's theorem and Young's inequality give
\[
 A^q\le(1+\varepsilon)b^q
 +C_p\varepsilon^{-1}b^{q-2}(A-b)^2.
\]
In this range
$x_0^{p-2}+x_1^{p-2}\le C_pb^{2-q}$, because $p-2<0$.
This is exactly \eqref{eq:tree-two-edge-residual}.
\end{proof}

\subsection{The relaxation gap on a tree}

\begin{thm}[Tree relaxation gap]\label{thm:tree-gap}
For the fixed exponent $1<p<\infty$,
\begin{equation}\label{eq:tree-gap}
 \lambda_p(G)\ge
 \begin{cases}
 c_p[T_G\log(en)^{1/(p-1)}]^{-1},&1<p<2,\\
 (2T_G)^{-1},&p=2,\\
 c_pT_G^{-1},&p>2.
 \end{cases}
\end{equation}
\end{thm}

\begin{proof}
Choose a root $r$ attaining the minimum in \eqref{eq:T-root}.  We first treat
$p>2$.  For a nonroot vertex $v$, let $T_v$ be its rooted subtree and put
\[
 Q_v=\sum_{u\in T_v}\cD_p(f,u).
\]
We claim, by induction upward from the leaves, that
\begin{equation}\label{eq:tree-superquadratic-induction}
 \frac{|f(v)-f(\operatorname{parent}v)|^p}{w_r(v)}\le C_pQ_v.
\end{equation}
At a leaf, $w_r(v)=1$, and updating $v$ deletes its sole incident-edge
energy, so \eqref{eq:tree-superquadratic-induction} holds.  If $v$ is
internal, apply Lemma~\ref{lem:tree-normalized-superquadratic} with
\[
 a_0=f(v)-f(\operatorname{parent}v),\qquad
 a_i=f(v)-f(z_i),\qquad w_i=w_r(z_i),
\]
where $z_1,\ldots,z_k$ are the children of $v$.  The recursion
\eqref{eq:tree-weight} supplies the required value $w=w_r(v)$, and the local
quantity $\mathfrak D_p(a)$ is exactly $\cD_p(f,v)$.  The induction hypothesis
at the children therefore proves \eqref{eq:tree-superquadratic-induction}.

Multiplying \eqref{eq:tree-superquadratic-induction} by $w_r(v)$, summing over
the nonroot vertices, and reversing the subtree sums gives
\begin{align*}
 \cE_p(f)
 &\le C_p\sum_{v\ne r}w_r(v)Q_v\\
 &=C_p\sum_{u\in V}\cD_p(f,u)
 \sum_{\substack{v\in\operatorname{Path}(u,r)\\v\ne r}}w_r(v)
 \le C_pT_G\sum_{u\in V}\cD_p(f,u).
\end{align*}
No local estimate at the root is needed; adding its nonnegative drop only
enlarges the final sum.  This proves the third line of \eqref{eq:tree-gap}.

Now let $1<p<2$, write $a=p-1$, and put $L_n=\log(en)$.  For a nonroot
vertex $u$, let $x_0(u)$ be the magnitude of the increment on its parent edge,
let $x_z(u)$ be the magnitude on the edge to a child $z$, and put
$\rho_u=|B_f(u)|$.  If $u$ has children, define
\begin{equation}\label{eq:tree-alpha}
 \alpha_{u,z}=\left(\frac{w_r(z)+1}{w_r(u)}\right)^{p-1}.
\end{equation}
The recursion \eqref{eq:tree-weight} says that these numbers sum to one.  At a
leaf put $\varepsilon_u=0$.  Since $w_r(u)=1$ and updating $u$ removes its
only incident-edge energy, $\cD_p(f,u)=x_0(u)^p$; thus the local estimate
below holds directly, with an empty child sum.  At a unary internal vertex
put $\varepsilon_u=1/w_r(u)$ and use
Lemma~\ref{lem:tree-two-edge-residual}.  At a branching vertex put
\begin{equation}\label{eq:tree-branch-epsilon}
 \varepsilon_u=\kappa_p\frac{\deg(u)}{w_r(u)^aL_n},
\end{equation}
where $\kappa_p>0$ will be fixed below.

We first record the branch increment that makes this choice admissible.  Set
$z_v=w_r(v)^a$.  If $u$ has $k\ge2$ children and $y$ is any one of them,
then $\deg(u)=k+1$ and every side child has weight at least one.  Therefore
\begin{align}\label{eq:branch-increment}
 z_u-z_y
 &= (w_r(y)+1)^a-w_r(y)^a
   +\sum_{\substack{z\text{ child of }u\\z\ne y}}(w_r(z)+1)^a\notag\\
 &\ge 2^a(k-1),
 \qquad
 \deg(u)\le 3\cdot2^{-a}(z_u-z_y).
\end{align}
In particular, $\deg(u)\le C_pz_u$.  Since the tree is nontrivial,
$n\ge2$ and $L_n\ge1$; choosing $\kappa_p$ sufficiently small therefore
ensures $0<\varepsilon_u\le1$ at every branching vertex.

Use Lemma~\ref{lem:tree-power-residual} at branching vertices.  Together
with Lemma~\ref{lem:tree-two-edge-residual} at unary vertices, the first line
of Lemma~\ref{lem:local-drop-comparison}, and the preceding leaf observation,
the three vertex types give
\begin{equation}\label{eq:tree-local-subquadratic}
 x_0(u)^p\le
 (1+\varepsilon_u)\sum_{z\text{ child}}
 \frac{w_r(u)}{w_r(z)+1}x_z(u)^p
 +C_pw_r(u)L_n^{1/a}\cD_p(f,u).
\end{equation}
Indeed, at a branching vertex the residual coefficient is
\[
 C_p\deg(u)^{1/a}\varepsilon_u^{-1/a}
 =C_p\kappa_p^{-1/a}w_r(u)L_n^{1/a},
\]
whereas at a unary vertex the sharper lemma gives $C_pw_r(u)$, which is
absorbed because $L_n\ge1$.  At a leaf the assertion is the direct identity
recorded above.

We make the iteration of \eqref{eq:tree-local-subquadratic} explicit.  For
$u\ne r$, set
\[
 e_u=x_0(u)^p,\qquad
 R_u=C_pL_n^{1/a}w_r(u)\cD_p(f,u),
\]
and, whenever $z$ is a child of $u$, set
\begin{equation}\label{eq:tree-edge-transfer}
 A_{u,z}=(1+\varepsilon_u)\frac{w_r(u)}{w_r(z)+1}.
\end{equation}
Since $x_z(u)=x_0(z)$, after choosing one uniform value of $C_p$ the local
estimate is the recursion
\begin{equation}\label{eq:tree-energy-recursion}
 e_u\le R_u+\sum_{z\text{ child of }u}A_{u,z}e_z.
\end{equation}
For $v$ in the rooted subtree $T_s$, define the transfer coefficient
$\Gamma(s,v)$ by
\[
 \Gamma(s,s)=1,\qquad
 \Gamma(s,v)=A_{s,z}\Gamma(z,v),
\]
where $z$ is the first vertex after $s$ on the path from $s$ to $v$.
Induction on the height of $T_s$ now gives the precise iterated inequality
\begin{equation}\label{eq:tree-energy-iteration}
 e_s\le\sum_{v\in T_s}\Gamma(s,v)R_v.
\end{equation}
Indeed, this is \eqref{eq:tree-energy-recursion} when $s$ is a leaf; at an
internal vertex, substituting the induction hypothesis for every child gives
\[
 e_s\le R_s+\sum_{z\text{ child of }s}A_{s,z}
       \sum_{v\in T_z}\Gamma(z,v)R_v
     =\sum_{v\in T_s}\Gamma(s,v)R_v,
\]
by the defining recursion for $\Gamma$.  Thus every occurrence of
$\cD_p(f,v)$ has acquired exactly the transfers along the unique path from
its ancestor $s$ to $v$.

It remains to show that these path products are uniformly bounded.  First
note that
$(x+1)^a\le x^a+1$.  An induction on the size $N_v$ of the rooted subtree
below $v$ consequently gives
\begin{equation}\label{eq:tree-z-size}
 z_v\le2N_v-1\le2n.
\end{equation}
Now fix a segment of an ancestor path containing only nonroot vertices and,
at each branching vertex $u$ on it, let $y$ be the child that continues along
the path.  The values of $z$ strictly increase
when the path is followed toward the root.  Moreover, the intervals
$[z_y,z_u]$ belonging to successive branch points are ordered and have
disjoint interiors:
between two such points the path is unary, and at every unary step
$w_r(\operatorname{parent})=w_r(\operatorname{child})+1$.
Using \eqref{eq:branch-increment} and
$(z_u-z_y)/z_u\le\log(z_u/z_y)$, we obtain
\begin{align}\label{eq:branch-integral}
 \sum_{u\text{ branching on the path}}\varepsilon_u
 &\le \frac{C_p\kappa_p}{L_n}
 \sum_{u\text{ branching on the path}}\frac{z_u-z_y}{z_u}\notag\\
 &\le \frac{C_p\kappa_p}{L_n}
 \sum_{u\text{ branching on the path}}\log\frac{z_u}{z_y}
 \le \frac{C_p\kappa_p\log(2n)}{L_n}
 \le C_p\kappa_p.
\end{align}
Make $\kappa_p$ smaller, if necessary, so that the last quantity is at most
$1/2$.  This choice simultaneously retains $\varepsilon_u\le1$ and, using
$1+x\le e^x$, gives the branch-only product bound
\begin{equation}\label{eq:tree-branch-product}
 \prod_{u\text{ branching on the path}}(1+\varepsilon_u)\le e^{1/2}.
\end{equation}
Let $s$ be a proper ancestor of $v$, and write the ancestor--descendant path
as $s=u_0,u_1,\ldots,u_m=v$, where $m\ge1$.  The recursive definition of
$\Gamma$, together with \eqref{eq:tree-edge-transfer}, gives
\begin{align}\label{eq:tree-sub-product}
 \Gamma(s,v)
 &=\prod_{i=0}^{m-1}A_{u_i,u_{i+1}}\notag\\
 &=\frac{w_r(s)}{w_r(v)+1}
 \prod_{i=1}^{m-1}\frac{w_r(u_i)}{w_r(u_i)+1}
 \prod_{i=0}^{m-1}(1+\varepsilon_{u_i}).
\end{align}
For every internal unary vertex $u_i$, $1\le i<m$, the recursion gives
$w_r(u_i)=w_r(u_{i+1})+1$, and its ratio and perturbation factor cancel:
\[
 \frac{w_r(u_i)}{w_r(u_i)+1}
 \left(1+\frac1{w_r(u_i)}\right)=1.
\]
If the starting vertex $s$ is unary, its perturbation factor is instead
absorbed at the left endpoint:
\[
 w_r(s)\left(1+\frac1{w_r(s)}\right)=w_r(s)+1.
\]
If $s$ is branching, simply use $w_r(s)\le w_r(s)+1$.  All uncancelled
internal ratios are at most one, and
\eqref{eq:tree-branch-product} controls precisely the remaining branching
factors.  Hence
\begin{equation}\label{eq:tree-transfer-bound}
 \Gamma(s,v)\le
 e^{1/2}\frac{w_r(s)+1}{w_r(v)+1}.
\end{equation}

The diagonal coefficient is $\Gamma(v,v)=1$.  Summing
\eqref{eq:tree-energy-iteration} over all nonroot ancestors $s$ and reversing
the finite sums now makes the accumulation completely explicit:
\begin{align}
 \cE_p(f)=\sum_{s\ne r}e_s
 &\le C_pL_n^{1/a}
 \sum_{s\ne r}\sum_{v\in T_s}
 \Gamma(s,v)w_r(v)\cD_p(f,v)\notag\\
 &=C_pL_n^{1/a}\sum_{v\ne r}w_r(v)\cD_p(f,v)
 \sum_{\substack{s\in\operatorname{Path}(v,r)\\s\ne r}}
 \Gamma(s,v).\label{eq:tree-reversed-iteration}
\end{align}
For a fixed $v$, the diagonal term and \eqref{eq:tree-transfer-bound} give
\begin{align*}
 w_r(v)\sum_{\substack{s\in\operatorname{Path}(v,r)\\s\ne r}}
 \Gamma(s,v)
 &\le w_r(v)+e^{1/2}
 \sum_{\substack{s\in\operatorname{Path}(v,r)\\s\ne r,\ s\ne v}}
 (w_r(s)+1)\\
 &\le e^{1/2}
 \sum_{\substack{s\in\operatorname{Path}(v,r)\\s\ne r}}(w_r(s)+1)
 \le 2e^{1/2}T_G,
\end{align*}
where the ratio $w_r(v)/(w_r(v)+1)$ was bounded by one and every tree
weight is at least one.  Applying this bound in
\eqref{eq:tree-reversed-iteration}, and then adding the nonnegative root
drop, proves the first line of \eqref{eq:tree-gap}.  Notice that
the finite-size estimate \eqref{eq:branch-integral}, rather than a pointwise
comparison with one side branch, is what controls the full branch product.

Finally take $p=2$ and choose a root attaining $T_G$.  Weighted
Cauchy--Schwarz along the unique path from $u$ to $r$, followed by summation
in $u$, gives the path estimate of Diaconis and
Stroock~\cite{diaconis1991geometric} in the explicit form
\begin{align}\label{eq:tree-quadratic-path}
 \sum_{u\in V}\deg(u)|f(u)-f(r)|^2
 &\le T_G\sum_{v\ne r}
 \frac{\operatorname{vol}(T_v)}{w_r(v)}
 |f(v)-f(\operatorname{parent}v)|^2,
\end{align}
where $\operatorname{vol}(T_v)=\sum_{u\in T_v}\deg(u)$.  Indeed, the first
Cauchy--Schwarz factor is at most $T_G$, and reversing the resulting path
sums makes the coefficient of the edge above $v$ equal to
$\operatorname{vol}(T_v)/w_r(v)$.

If $s_v=|T_v|$, then $\operatorname{vol}(T_v)=2s_v-1\le2s_v$.  The
$p=2$ recursion also gives $w_r(v)\ge s_v$: this is equality at a leaf, and
if $v$ has $k\ge1$ children, induction gives
\[
 w_r(v)=\sum_{z\text{ child of }v}(w_r(z)+1)
 \ge\sum_z(s_z+1)=s_v+k-1\ge s_v.
\]
Consequently \eqref{eq:tree-quadratic-path} is at most
$2T_G\cE_2(f)$.  Taking the infimum over constants in the numerator shows
$\Phi_2(G)\le2T_G$, hence $\mu_1(G)=\Phi_2(G)^{-1}\ge(2T_G)^{-1}$.
The quadratic identity $\lambda_2(G)=\mu_1(G)$ recorded in
Section~\ref{sec:upper} proves the middle line of \eqref{eq:tree-gap}.
\end{proof}

\begin{cor}[Tree consensus upper bound]\label{cor:tree-consensus-upper}
For $f_0\in[-1,1]^V$ and $0<\epsilon\le1$,
\begin{equation}\label{eq:tree-consensus-upper}
 \E\tau_p(\epsilon;f_0)\le C_pnT_G\log(n/\epsilon)
 \begin{cases}
  \log(en)^{1/(p-1)},&1<p<2,\\
  1,&p\ge2.
 \end{cases}
\end{equation}
\end{cor}

\begin{proof}
Insert Theorem~\ref{thm:tree-gap} into
Theorem~\ref{thm:variational-upper}; on a tree, $|E|=n-1$ and
$\diam(G)<n$.
\end{proof}

\subsection{A range-one slow profile}

\begin{thm}[Tree lower bound]\label{thm:tree-lower}
For every $1<p<\infty$, there is $f_0\in[0,1]^V$ of range exactly one such
that
\begin{equation}\label{eq:tree-lower}
 \E\tau_p(1/2;f_0)\ge c_p\frac{nT_G}{\Delta}.
\end{equation}
\end{thm}

\begin{proof}
For a root $r$, abbreviate
\[
 \delta_r(x)=\sum_{v\in\operatorname{Path}(x,r),\,v\ne r}w_r(v).
\]
Every maximizer of $\delta_r$ is a graph-theoretic leaf.  Indeed,
$\delta_r(r)=0$, whereas a neighbor of $r$ has positive value, so a maximizer
$x$ is not the root.  If $x$ had a child $y$, then
$\delta_r(y)=\delta_r(x)+w_r(y)>\delta_r(x)$, a contradiction.  Thus $x$ is
a nonroot vertex with no children and hence has graph degree one.

Fix a total order to break ties among maximizers.  Choose a
graph-theoretic leaf $s$ and let $t$ be the selected maximizer of
$\delta_s$.  Move
the root along the path from $s$ to $t$.  At the root $s$, the selected
maximizer is $t$, whereas at the root $t$ it is not.  Hence there is a first
adjacent pair $a,b$, in this order along the path, for which $t$ is selected
at root $a$ but not at root $b$.  Let $z$ be the selected maximizer at root
$b$.

Remove the edge $ab$, and let $C_b$ and $C_a$ be the components containing
$b$ and $a$.  For $x\in C_b$, the orientations and all weights strictly
below $b$ are unchanged by rerooting across $ab$, and
\[
 \delta_a(x)=w_a(b)+\delta_b(x).
\]
Thus $t$ remains the tie-broken maximizer within $C_b$.  Since it is not the
global selected maximizer for root $b$, the vertex $z$ lies in $C_a$.

Define
\[
 c_b=\min\{1,w_b(a)/w_a(b)\},\qquad
 c_a=\min\{1,w_a(b)/w_b(a)\},
\]
so that $c_bw_a(b)=c_aw_b(a)=:h$.  Construct an unnormalized profile $g$ as
follows.  Set $g(a)=0$ and $g(b)=-h$.  In $C_b$, orient away from $a$ and
decrease by $c_bw_a(v)$ across the edge entering $v$.  In $C_a$, orient away
from $b$ and increase by $c_aw_b(v)$ across the edge entering $v$.  The
definition gives the two central identities
\begin{equation}\label{eq:tree-central-identities}
 \frac{h}{c_a}=w_b(a),\qquad \frac{h}{c_b}=w_a(b).
\end{equation}
More generally, the signs in the two components are explicit:
\[
 g(x)=-c_b\delta_a(x)\quad(x\in C_b),\qquad
 g(x)=-h+c_a\delta_b(x)\quad(x\in C_a).
\]
Every path from $C_b$ to the root $a$ contains $b$, whereas every path from
$C_a$ to the root $b$ contains $a$.  Thus
\[
 g(x)\le-h\quad(x\in C_b),\qquad g(x)\ge0\quad(x\in C_a),
\]
by \eqref{eq:tree-central-identities}.  Since $t$ maximizes $\delta_a$ and
$z$ maximizes $\delta_b$, we have $g(t)\le g(x)$ on $C_b$ and
$g(z)\ge g(x)$ on $C_a$.  The sign separation across the central edge then
shows that $g(t)$ and $g(z)$ are respectively the global minimum and maximum.
Hence the range of $g$ is
\begin{equation}\label{eq:tree-profile-range}
 A=c_b\delta_a(t)+c_a\delta_b(z)-h.
\end{equation}
One of $c_a,c_b$ is one.  Moreover,
\[
 c_b\delta_a(t)\ge c_bw_a(b)=h,
 \qquad c_a\delta_b(z)\ge c_aw_b(a)=h.
\]
If $c_b=1$, the second inequality gives
$A\ge\delta_a(t)=T_G(a)$; if $c_a=1$, the first gives
$A\ge\delta_b(z)=T_G(b)$.  In either case $A\ge T_G$.

Let $f_0=(g-\min g)/A$.  Its range is exactly one.  We claim that
\begin{equation}\label{eq:tree-profile-psi}
 \lVert Sf_0-f_0\rVert_\infty\le A^{-1}\le T_G^{-1}.
\end{equation}
Indeed, translation equivariance and the positive homogeneity of the scalar
$p$-mean from Lemma~\ref{lem:mean-properties} allow us to work with incident
increments and divide out any common positive scale.  First
consider a noncentral vertex in one of the two oriented components and
divide out the common factor $c_a/A$ or $c_b/A$.  Its
parent-edge increment has magnitude
$w(v)$, while its child-edge increments have magnitudes $w(z)$.  If the local
$p$-mean moves toward the parent by $x$, its first-order equation is
\[
 (w(v)-x)^{p-1}=\sum_{z\text{ child}}(w(z)+x)^{p-1}.
\]
At $x=0$ the left side is at least the right, while at $x=1$ the recursion
\eqref{eq:tree-weight} makes it at most the right.  The zero therefore lies
in $[0,1]$.  A leaf moves exactly one unit toward its parent.  At $a$, after
dividing out the factor $c_a/A$, the central-edge magnitude is
$h/c_a=w_b(a)$ and the child-edge magnitudes are the corresponding
$w_b$-weights.  At $b$, after dividing out $c_b/A$, they are
$h/c_b=w_a(b)$ and the corresponding $w_a$-weights.  Thus
\eqref{eq:tree-central-identities} makes the same recursion valid at both
central vertices.  Restoring the scale proves \eqref{eq:tree-profile-psi}.

Theorem~\ref{thm:propagation-lower} now proves \eqref{eq:tree-lower}.
\end{proof}

\begin{cor}[Worst-profile tree lower bound]
\label{cor:tree-worst-lower}
For every $1<p<\infty$,
\begin{equation}\label{eq:tree-worst-lower}
 \mathsf T_p(G,1/2)
 \ge c_p n\max\left\{\log n,\frac{T_G}{\Delta}\right\}.
\end{equation}
The two terms in the maximum need not be realized by the same profile.
\end{cor}

\begin{proof}
Combine the coupon observation in Section~\ref{sec:lower} with
Theorem~\ref{thm:tree-lower}.  Both profiles lie in $[0,1]^V$, so their lower
bounds also apply to the globally defined supremum over $[-1,1]^V$.  Use
$\max\{a,b\}\ge(a+b)/2$ to absorb the two absolute constants.
\end{proof}

\subsection{Spherically symmetric trees}

Consider layers $0,1,\ldots,k$.  Assume $d_i\ge1$ for every $0\le i<k$,
so all displayed layers are nonempty.  Every vertex in layer $i$ has $d_i$
children, and $o$ denotes the layer-zero root.

\begin{prop}[Symmetric-tree geometry]
\label{prop:symmetric-tree-geometry}
Assume $d_0\ge2$.  Then $o$ minimizes $T_G(r)$.  If
\begin{equation}\label{eq:symmetric-S}
 S=\sum_{1\le i\le j\le k-1}
 (d_i d_{i+1}\cdots d_j)^{1/(p-1)},
\end{equation}
where the sum is empty for $k=1$, then
\begin{equation}\label{eq:symmetric-T-comparison}
 1+S\le T_G\le1+2S.
\end{equation}
The actual maximum degree is
\begin{equation}\label{eq:symmetric-degree}
 \Delta_G=\max\left\{d_0,\ 1+\max_{1\le i<k}d_i\right\},
\end{equation}
with the second maximum omitted for $k=1$.
\end{prop}

\begin{proof}
Let $W_i$ be the common natural-root weight in layer $i$.  Then
\begin{equation}\label{eq:symmetric-recurrence}
 W_k=1,\qquad W_i=d_i^{1/(p-1)}(W_{i+1}+1)\quad(1\le i<k),
\end{equation}
and $T_G(o)=\sum_{i=1}^kW_i$.  If $r\ne o$, it lies in one of the
$d_0$ first-level branches.  Choose a deepest leaf in a different branch.
Along the part of its path leading to $o$, all orientations and weights agree
with the natural rooting.  Hence $T_G(r)\ge\sum_{i=1}^kW_i=T_G(o)$.

Expanding \eqref{eq:symmetric-recurrence} gives
\[
 W_i=\sum_{j=i}^{k-1}(d_i\cdots d_j)^{1/(p-1)}
 +(d_i\cdots d_{k-1})^{1/(p-1)},\qquad 1\le i<k.
\]
After summing over $i$ and adding $W_k=1$, the first double sum is $S$ and
the terminal products form a subcollection of the terms of $S$.  This proves
\eqref{eq:symmetric-T-comparison}.  Formula
\eqref{eq:symmetric-degree} follows by distinguishing the root, internal
layers, and leaves.
\end{proof}

\begin{cor}\label{cor:symmetric-tree-consensus}
Under the assumptions of Proposition~\ref{prop:symmetric-tree-geometry}, the
upper bound in Corollary~\ref{cor:tree-consensus-upper} and the lower bounds
in Theorem~\ref{thm:tree-lower} and Corollary~\ref{cor:tree-worst-lower} hold
with $T_G$ replaced, up to an absolute factor two, by $1+S$ and with
$\Delta$ given by \eqref{eq:symmetric-degree}.
\end{cor}

\begin{rem}[Why unary stems are excluded]\label{rem:unary-stem}
When $d_0=1$, the natural root need not minimize \eqref{eq:T-root}, and the
one-sided product formula must not be used.  For example, at $p=2$ the
branching sequence $(1,1,1,3,3)$ has natural-root value $73$.  If $r'$ is the
first branching vertex, a direct calculation gives $T_G(r')=7$, and hence
$T_G\le7<73$; this is all that is needed to show that the natural root is not
a minimizer.  The general theorem remains valid because its definition
minimizes over all roots.
\end{rem}

\section{Conductance expanders}\label{sec:expanders}

For $A\subseteq V$, write
$\operatorname{vol}(A)=\sum_{u\in A}\deg(u)$ and let $\partial A$ be the
set of edges with exactly one endpoint in $A$.  The volume conductance is
\begin{equation}\label{eq:conductance}
 h(G)=\min_{0<\operatorname{vol}(A)\le\operatorname{vol}(V)/2}
 \frac{|\partial A|}{\operatorname{vol}(A)}.
\end{equation}

\begin{prop}[Scalar Poincar\'e inequality from conductance]
\label{prop:conductance-poincare}
For every $1<p<\infty$,
\begin{equation}\label{eq:conductance-Phi}
 \Phi_p(G)\le C_p h(G)^{-p}.
\end{equation}
Consequently,
\begin{equation}\label{eq:conductance-gap}
 \lambda_p(G)\ge c_p
 \begin{cases}
  h(G)^{p/(p-1)},&1<p<2,\\
  h(G)^2,&p\ge2.
 \end{cases}
\end{equation}
\end{prop}

\begin{proof}
Let $m$ be a median of $f$ for the degree measure, so that both
$\{f>m\}$ and $\{f<m\}$ have volume at most half the total.  Put
$g=(f-m)_+$.  The layer-cake formula, the conductance inequality applied to
$\{g>t\}$, and edge coarea give
\begin{align}
 h(G)\sum_u\deg(u)g(u)^p
 &\le\int_0^\infty pt^{p-1}|\partial\{g>t\}|\,dt\notag\\
 &=\sum_{\{u,v\}\in E}|g(u)^p-g(v)^p|.\label{eq:coarea-p}
\end{align}
For nonnegative $a,b$,
$|a^p-b^p|\le p|a-b|(a^{p-1}+b^{p-1})$.  H\"older's inequality on edges and
$(a^{p-1}+b^{p-1})^{p/(p-1)}\le C_p(a^p+b^p)$ therefore imply
\[
 \sum_{\{u,v\}\in E}|g(u)^p-g(v)^p|
 \le C_p\cE_p(g)^{1/p}
 \left(\sum_u\deg(u)g(u)^p\right)^{(p-1)/p}.
\]
If $g\equiv0$, \eqref{eq:positive-cheeger} below is immediate.  Otherwise,
$\sum_u\deg(u)g(u)^p>0$, so combining this estimate with
\eqref{eq:coarea-p} and dividing by its $(p-1)/p$ power gives
\begin{equation}\label{eq:positive-cheeger}
 \sum_u\deg(u)g(u)^p\le(C_p/h(G))^p\cE_p(g).
\end{equation}
Apply the same argument to $(m-f)_+$, with the inequality again immediate if
this truncation vanishes identically and division performed only otherwise.
On every edge, the sum of the two truncated $p$-energies is at most
$|f(u)-f(v)|^p$.  Adding the two copies of \eqref{eq:positive-cheeger} yields
\[
 \sum_u\deg(u)|f(u)-m|^p\le C_ph(G)^{-p}\cE_p(f).
\]
Since the infimum over constants is no larger than its value at the chosen
median,
\[
 \inf_{c\in\R}\sum_u\deg(u)|f(u)-c|^p
 \le \sum_u\deg(u)|f(u)-m|^p.
\]
This proves \eqref{eq:conductance-Phi}.  Equation
\eqref{eq:conductance-gap} follows from
Theorem~\ref{thm:nonlinear-poincare-comparison}.
\end{proof}

\begin{thm}[Consensus on conductance expanders]
\label{thm:expander-consensus}
Fix $1<p<\infty$.  Let $G$ have $n$ vertices and satisfy $h(G)\ge h_0>0$.
For every $f_0\in[-1,1]^V$ and $0<\epsilon\le1$,
\begin{equation}\label{eq:expander-upper}
 \E\tau_p(\epsilon;f_0)\le C_{p,h_0}n\log(n/\epsilon).
\end{equation}
Moreover, at tolerance $1/2$ the globally defined worst-profile consensus
time satisfies
\begin{equation}\label{eq:expander-lower}
 cn\log n
 \le \mathsf T_p(G,1/2)
 \le C_{p,h_0}n\log n.
\end{equation}
Thus, for fixed $p$ and $h_0$, the worst-profile consensus time at fixed
tolerance is exactly of order $n\log n$, with no degree assumption.
\end{thm}

\begin{proof}
Proposition~\ref{prop:conductance-poincare} gives a positive lower bound on
$\lambda_p(G)$ depending only on $p$ and $h_0$.  In
Theorem~\ref{thm:variational-upper}, the remaining logarithm is
$O_p(\log(n/\epsilon))$, proving \eqref{eq:expander-upper}.

The upper inequality in \eqref{eq:expander-lower} is
\eqref{eq:expander-upper} with $\epsilon=1/2$.  The lower inequality is the
coupon observation in Section~\ref{sec:lower}.
\end{proof}

\section{Conclusion and open problems}\label{sec:conclusion}

The exact relaxation gap gives a nonsingular upper-bound framework for all
three graph families.  Coupon collection and causal propagation provide most
of the lower bounds, while the sharp quadratic box lower bound uses the
spectral estimate for asynchronous DeGroot dynamics.  On boxes, the explicit
slow profiles, that spectral input, and the tensor Poincar\'e inequality
identify the polynomial scale up to logarithmic factors.  If logarithms are
suppressed, the causal-profile comparison can leave at most a factor $d$, and
the gap is smaller when the coupon term dominates.  On trees, the
rerooting-invariant quantity $T_G=T_G^{(p)}$ controls both the energy
telescoping and a range-one slow profile.  This gives a polylogarithmic
near-characterization for bounded-degree trees; for arbitrary-degree trees the
remaining factor $\Delta$ can be polynomial in $n$.  On conductance expanders,
the Poincar\'e constant is uniformly bounded and the coupon obstruction makes
the worst-profile fixed-tolerance consensus time exactly of order $n\log n$.

There is also an all-graph consequence beyond the three families.  For every
$2\le p<\infty$, Corollary~\ref{cor:universal-p-comparison} gives
\[
 \mathsf T_p(G,\epsilon)
 \le C_p\log\!\left(\frac{en}{\epsilon}\right)
       \mathsf T_2(G,\epsilon),
 \qquad 0<\epsilon\le\frac12.
\]
This compares independently maximized worst-profile times.  It is therefore a
partial answer to
\cite[Question~6.1]{amir2025convergencerateellpenergyminimization}, which asks
for a comparison for each common initial profile.

Both questions below concern the genuinely nonlinear dynamics and exclude
$p=2$, which has additional linear semigroup and spectral tools.  The first
question also includes the endpoint $p=\infty$.  Write
\[
  \mathsf T_p(G,\epsilon)
  :=\sup_{f_0\in[-1,1]^V}\E\tau_p(\epsilon;f_0).
\]

\begin{question}[Removing the maximum-degree loss]
\label{question:remove-delta}
Fix $p\in(1,\infty]\setminus\{2\}$.  Can the factor $\Delta$ be removed from
the propagation lower bound in Theorem~\ref{thm:propagation-lower}?  More
precisely, is there a constant $c_p>0$ such that, for every finite connected
graph, every nonconstant profile $f_0$, and every
$0<\epsilon\le\osc(f_0)/2$,
\begin{equation}
  \E\tau_p(\epsilon;f_0)
  \ge c_p\frac{n\osc(f_0)}{\lVert Sf_0-f_0\rVert_\infty}?
  \label{eq:conjectural-degree-free-propagation}
\end{equation}
The loss in the present proof occurs when all nearest-neighbor causal chains
are counted by a union bound.  A proof of
\eqref{eq:conjectural-degree-free-propagation} would have to average competing
influences without paying for the maximum number of possible next steps.

\end{question}

For the remainder, fix $p\in(1,\infty)\setminus\{2\}$.  Define the
unit-oscillation defect
\begin{equation}
  \kappa_p(G)
  :=\inf\bigl\{\lVert Sf-f\rVert_\infty:\osc(f)=1\bigr\}.
  \label{eq:kappa-p}
\end{equation}
After normalizing $\min f=0$ and $\max f=1$, compactness and the maximum
principle show that the infimum is attained and is positive.  Theorem
\ref{thm:propagation-lower} gives
\begin{equation}
  \mathsf T_p(G,1/2)
  \gtrsim_p \frac{n}{\Delta\kappa_p(G)},
  \label{eq:kappa-current-lower}
\end{equation}
whereas Theorem~\ref{thm:variational-upper} gives
\begin{equation}
  \mathsf T_p(G,1/2)
  \lesssim_p \frac{n}{\lambda_p(G)}\log(en).
  \label{eq:lambda-current-upper}
\end{equation}

\begin{question}[A universal nonlinear consensus parameter]
\label{question:universal-parameter}
For fixed $p\in(1,\infty)\setminus\{2\}$, is there an intrinsic graph
quantity that characterizes $\mathsf T_p(G,1/2)$ on every finite connected
graph, up to constants and logarithmic factors?  In particular, does
\begin{equation}
  \mathfrak C_p(G)
  :=n\kappa_p(G)^{-1}
  \label{eq:defect-condition-number}
\end{equation}
have this property?

\end{question}

The residual gap in the box comparison---at most a factor $d$ after
suppressing logarithms, and smaller when the coupon term dominates---and the
degree and polylogarithmic losses on trees are concrete tests for
Question~\ref{question:universal-parameter}.  An
affirmative answer to Question~\ref{question:remove-delta} would also remove
the maximum-degree loss from the general $p=\infty$ diameter lower bound.

\bibliographystyle{plain}
\bibliography{bibliography}

\end{document}